\documentclass[a4paper,11pt]{amsart}

\usepackage[utf8]{inputenc}

\makeatletter{}
\usepackage{amsmath, a4wide, amssymb, amsfonts, amsthm}
\usepackage{fullpage}
\usepackage{graphicx}

 \usepackage{caption}
\usepackage[labelformat=simple]{subcaption}

\usepackage{mathtools}

\DeclarePairedDelimiter\abs{\lvert}{\rvert}

\usepackage[  ]{todonotes}

\usepackage{tikz}
\usetikzlibrary{calc}\usetikzlibrary{decorations.pathmorphing}
\usetikzlibrary{decorations.markings}
\usetikzlibrary{backgrounds}

\makeatletter{}\newcommand{\startpicturesplicingproof}{	\renewcommand*{\locationOfNode}{1.5}
	\renewcommand*{\locationOfStart}{3}
	\renewcommand*{\heigthofedges}{1}
	\coordinate (dOne) at (-\locationOfStart,\heigthofedges);
	\coordinate (dN) at (-\locationOfStart,-\heigthofedges)  ;
	\coordinate (daOne) at (\locationOfStart,\heigthofedges)  ;
	\coordinate (daN) at (\locationOfStart,-\heigthofedges)  ;
	\coordinate (down) at (0,-\heigthofedges)  ;
	\coordinate (vL) at (-\locationOfNode,0) ;
	\coordinate (vR) at (\locationOfNode,0)  ;
	\coordinate (vLL) at (-\locationOfStart,0) ;
	\coordinate (vRR) at (\locationOfStart,0)  ;
	\coordinate (labelup) at (0,.3);
	\coordinate (labeldown) at (0,-.3);
		\coordinate (down) at (0,-\heigthofedges)  ;
	
	\coordinate (smalldown) at (0,-.1);
	\coordinate (smallup) at (0,.05);
	
	\coordinate (vLu) at ($(vL) + (smallup)$);
	\coordinate (vLLu) at ($(vLL) + (smallup)$);
	\coordinate (vRu) at ($(vR) + (smallup)$);
	\coordinate (vRRu) at ($(vRR) + (smallup)$);

	\coordinate (vLd) at ($(vL) + (smalldown)$);
	\coordinate (vLLd) at ($(vLL) + (smalldown)$);
	\coordinate (vRd) at ($(vR) + (smalldown)$);
	\coordinate (vRRd) at ($(vRR) + (smalldown)$);
	}
	\newcommand*{\widthoppaddingbackgroundsplicingproof}{12pt}
\newcommand{\partApicturesplicingproof}{
\begin{scope}[on background layer]
    \filldraw[color=\colorA!30, line width=\widthoppaddingbackgroundsplicingproof, rounded corners]   ($(dOne) + (-.5,.5)$) -- ($(dOne)+ (.5,.5) $) -- ($(vL)$) -- ($(dN) + (0,-.25)$) -- ($(dN) + (-.5,-.5)$) -- cycle;
 
 \end{scope}
\draw (dOne)--(vL) node [midway, above ]  {$d_1$};
\draw (dN)--(vL)  node [midway, above ]  {$d_n$};
\draw node at ($ 1/2*(dOne) + 1/2*(dN) +(1/3,0) $) {\vdots};
\fill ($1/4*(vL) + 3/4*(dOne)$) circle [radius=3pt];
\fill ($2/3*(vL) + 1/3*(dOne)$) circle [radius=3pt];
\fill ($1/4*(dN) + 3/4*(vL)$) circle [radius=3pt];
\fill ($2/3*(dN) + 1/3*(vL)$) circle [radius=3pt];
\draw node[color=\colorA, rectangle, rounded corners, draw=\colorA] at ($1/2*(vL) + 1/2*(dN)+(down)$){$A$} ;
}
\newcommand{\partBpicturesplicingproof}{
\begin{scope}[on background layer]
    \draw[color=\colorB!30, line width=\widthoppaddingbackgroundsplicingproof, rounded corners] (vL) -- (vR);
\end{scope}
\draw (vL)--(vR) node[very near start, above]{$d$} node[very near end, above]{$d'$};
\fill (vR) circle [radius=3pt];
\fill (vL) circle [radius=3pt];
\fill ($1/4*(vL) + 3/4*(vR)$) circle [radius=3pt];
\fill ($2/3*(vL) + 1/3*(vR)$) circle [radius=3pt];
\draw node[color=\colorB, rectangle, rounded corners, draw=\colorB] at ($1/2*(vL) + 1/2*(vR)+(down)$){$B$} ;
}
\newcommand{\partCpicturesplicingproof}{
\begin{scope}[on background layer]
\filldraw[color=\colorC!30,line width=\widthoppaddingbackgroundsplicingproof, rounded corners] 
 ($(vR) + (0,-0.5)$) --($(vR) + (0,.5)$) -- (daOne) -- (daN)   -- cycle; 
\end{scope}
\draw (vR)--(daOne)  node [midway, above ]  {$d'_1$};
\draw (vR)--(daN)  node [midway, above ]  {$d'_{n'}$};
\draw node at ($ 1/2*(daOne) + 1/2*(daN) +(-1/3,0) $) {\vdots};
\fill ($1/4*(vR) + 3/4*(daOne)$) circle [radius=3pt];
\fill ($2/3*(vR) + 1/3*(daOne)$) circle [radius=3pt];
\fill ($1/4*(daN) + 3/4*(vR)$) circle [radius=3pt];
\fill ($2/3*(daN) + 1/3*(vR)$) circle [radius=3pt];
\draw node[color=\colorC, rectangle, rounded corners, draw=\colorC] at ($1/2*(vR) + 1/2*(daN)+(down)$){$C$} ;
}
\newcommand{\partDpicturesplicingproof}{
\begin{scope}[on background layer]
	\filldraw[color=\colorD!30,line width=\widthoppaddingbackgroundsplicingproof, rounded corners]
	 ($(vLL) + (-1,-.2)$) -- ($(vLL) + (-1,.2)$)  -- ($(vL)+(0,1)$)-- ($(vL)+(0,-1)$) --cycle;
\end{scope}
\draw node[color=\colorD, rectangle, rounded corners, draw=\colorD] at ($1/2*(vL) + 1/2*(dN)+(down)$){$D$};
\draw[<-] (vLLu) -- (vLu) node[ near end, above]{$D$};
\draw[<-, dashed] (vLLd) -- (vLd);
\node at ($(vLLu) + (labelup)$) {$(M')$};
\node at ($(vLLd) + (labeldown)$) {$i'-1$};
\fill ($2/3*(vL) + 1/3*(vLL)$) circle [radius=3pt];
}
\newcommand{\partEpicturesplicingproof}{
\begin{scope}[on background layer]
	\filldraw[color=\colorE!30,line width=\widthoppaddingbackgroundsplicingproof, rounded corners] 
	 ($(vRR) + (1,-.2)$) -- ($(vRR) + (1,.2)$)  --  ($.5*(vRR) + .5*(vR) + (0,1.5)$)  -- ($(vR)+(0,.75)$)-- ($(vR)+(0,-.75)$) --cycle;
\end{scope}
	\draw[->] (vRu) -- (vRRu) node[ near start, above]{$D'$};
	\draw[->, dashed] (vRd) -- (vRRd);
	
	\node at ($(vRRd) + (labeldown)$) {$i-1$};
	\node at ($(vRRu) + (labelup)$) {$(M)$};
	\fill ($1/4*(vR) + 3/4*(vRR)$) circle [radius=3pt];
	\draw node[color=\colorE, rectangle, rounded corners, draw=\colorE] at ($1/2*(vR) + 1/2*(daN)+(down)$){$E$} ;
}

	\newcommand*{\heigthofedges}{1}
	\newcommand*{\locationOfNode}{2}
	\newcommand*{\locationOfStart}{4}
	
	\colorlet{darkbrown}{brown!50!black}
\newcommand*{\colorA}{red}
\newcommand*{\colorB}{green}
\newcommand*{\colorC}{blue}
\newcommand*{\colorD}{black}
\newcommand*{\colorE}{darkbrown}

\newcommand{\drawLetter}[2]{{\begin{tikzpicture}\draw node[color=#1, rectangle, rounded corners,inner sep=2,outer sep=0, draw=#1] at (0,0){\footnotesize  #2} ;\end{tikzpicture}}}

\newcommand{\letterA}{\drawLetter{\colorA}{A}{}}
\newcommand{\letterB}{\drawLetter{\colorB}{B}{}}
\newcommand{\letterC}{\drawLetter{\colorC}{C}{}}
\newcommand{\letterD}{\drawLetter{\colorD}{D}{}}
\newcommand{\letterE}{\drawLetter{\colorE}{E}{}}

\newcommand{\C}{\ensuremath{\mathbb{C}}}
\newcommand{\Q}{\ensuremath{\mathbb{Q}}}
\newcommand{\N}{\ensuremath{\mathbb{N}}}
\newcommand{\Z}{{\ensuremath{\mathbb{Z}}}}

\newcommand{\CompletedGrothendieckRing}[1]{\ensuremath{\hat{\mathcal{M}}_{#1}}}
\newcommand{\LocalizedGrothendieckring}[1]{\ensuremath{\mathcal{M}_{#1}}}
\newcommand{\Grothendieckring}[1]{\ensuremath{K_0(\Var_{{#1}})}}
\newcommand{\lefschetzmotive}{\ensuremath{\mathbb{L}}}
\newcommand{\classof}[1]{\ensuremath{ [#1]}}
\newcommand{\quotientInGrothendieckRing}[2]{\frac{#1}{#2}}

\newcommand{\Grothendieckgroupofabelianschemes}[1]{\ensuremath{\mathcal{A}_{#1}}}
\newcommand{\Grothendieckgroupofabelianvarieties}[1]{\ensuremath{{A}_{#1}}}
\newcommand{\Grothendieckgroupofabelianvarietiesuptoisogeny}[1]{\ensuremath{{A^0_{#1}}}}

\newcommand{\monodromicGrothendieckRing}[1]{\ensuremath{K_0^{\muhat}(\Var_{{#1}})}}
\newcommand{\monodromicLocalizedGrothendieckring}[1]{{\ensuremath{\mathcal{M}^{\muhat}_{#1}}}}
\newcommand{\monodromicCompletedGrothendieckRing}[1]{\ensuremath{\hat{\mathcal{M}}^{\muhat}_{#1}}}

\newcommand{\muhat}{\ensuremath{\hat{\mu }}}

\newcommand{\topZeta}[2]{Z_{#1}^{\text{top}}(#2)}
\newcommand{\topZetaTwisted}[4]{Z_{#1,#2}^{\text{top}, (#3)}(#4)}

\newcommand{\topZetaTwistedDiagram}[3]{Z_{#1}^{\text{top}, (#2)}(#3)}

\newcommand{\motivicZeta}[1]{Z_{#1}(T)}
\newcommand{\motivicZetaWithDifferentialForm}[2]{Z_{#1, #2}(T)}

\newcommand{\motivicZetaWithDifferentialFormT}[3]{Z_{#1, #2}(#3)}

\newcommand{\motivicZetaFunctioncoef}[1]{\ensuremath{\mathfrak{X}_{#1}}}
\newcommand{\monodromicmotivicZetaFunctioncoef}[1]{\ensuremath{\mathfrak{X}_{#1,1}}}
\newcommand{\monodromicmotivicZetaFunction}[1]{{\ensuremath{Z^{\muhat}_{#1}(T)}}}

\newcommand{\monodromicmotivicZetaFunctionT}[2]{{\ensuremath{Z^{\muhat}_{#1}(#2)}}}

\newcommand{\arcsoforderandac}[1]{\ensuremath{\mathcal{Z}_{#1,1}}}
\newcommand{\arcsoforder}[1]{\ensuremath{\mathcal{Z}_{#1}}}
\newcommand{\arcsoforderagainstdifform}[1]{\ensuremath{\Delta_{#1}}}

\newcommand{\EI}{\ensuremath{E_I}}
\newcommand{\EIO}{\ensuremath{E_I^\circ}}
\newcommand{\EIOT}{\ensuremath{\tilde{E}_I^\circ}}

\newcommand{\chitop}{\ensuremath{\chi_{\text{top}}}}

\newcommand{\njets}[2]{\ensuremath{\mathcal{L}_{#1} (#2)}}
\newcommand{\arcspace}[1]{\ensuremath{\mathcal{L} (#1)}}
\DeclareMathOperator{\ord}{ord}

\DeclareMathOperator\Hom{Hom}
\DeclareMathOperator\supp{Supp}

\DeclareMathOperator\Set{Set}

\newcommand{\inverselimit}{\ensuremath{\varprojlim}}

\DeclareMathOperator\Spec{Spec}
\DeclareMathOperator\divisor{div}
\DeclareMathOperator\Var{Var}
\DeclareMathOperator\Sch{Sch}
\DeclareMathOperator\Pic{Pic}
\DeclareMathOperator\Bl{Bl}
\DeclareMathOperator\NS{NS}
\DeclareMathOperator\dimensionop{dim}
\newcommand{\projectivenspace}[2]{\ensuremath{\mathbb{P}^{#1}_{#2}}}
\newcommand{\affinenspace}[2]{\ensuremath{\mathbb{A}^{#1}_{#2}}}
\newcommand{\dimension}[1]{\ensuremath{\dimensionop #1}}
\newcommand{\NeronSeveriGroup}[1]{\ensuremath{\NS(#1)}}

\newcommand{\sumofabelian}{\oplus}

\newcommand{\cohomologygroups}[3]{\ensuremath{H^{#1}(#2,#3)}}
\newcommand{\WeilintermediateJacobian}[2]{\ensuremath{A_{#2}^{0,#1}}}
\newcommand{\componentgroup}[2]{\ensuremath{A_{#2}^{c,#1}}}
\newcommand{\sumofhodgestructures}{\oplus}
\newcommand{\bigsumofhodgestructures}{\bigoplus}

\newcounter{general}[section]

\newtheorem{theorem}[general]{Theorem}
\newtheorem{lemma}[general]{Lemma}
\newtheorem{corollary}[general]{Corollary}
\newtheorem*{corollary*}{Corollary}
\newtheorem*{theorem*}{Theorem}
\newtheorem{proposition}[general]{Proposition}

\theoremstyle{definition}
\newtheorem*{definition}{Definition}

\newtheorem{example}[general]{Example}

\title{Splicing for motivic zeta functions}
 \author{Thomas Cauwbergs}
\address{KU Leuven, Dept. Wiskunde, Celestijnenlaan 200B, 3001 Leuven, Belgium}
\email{thomas.cauwbergs@wis.kuleuven.be}

\begin{document}
\begin{abstract}
 \noindent We lift the splicing formula of N\'emethi and Veys, which deals with polynomials in two variables, to the motivic level.
 After defining  the motivic zeta function   and 
 the monodromic motivic zeta function with respect to a differential form, we prove a splicing formula for 
 them, which specializes to this  formula of N\'emethi and Veys.
  We also show that we cannot introduce a monodromic motivic zeta functions in terms of
 a (splice) diagram since it does not contain all the necessary information.
 In the last part we discuss the generalized monodromy conjecture of N\'emethi and Veys. 
 The statement also holds for     motivic zeta functions but it turns out that the analogous statement 
  for   monodromic motivic zeta functions is not correct. We show some examples illustrating this.

\end{abstract}
\maketitle

\makeatletter{}\section*{Introduction}
We will consider here  a polynomial $f\in \C[x,y]$, which has at most a singularity at the origin.

\subsection*{Splicing}
The topology of a plane curve singularity $\{f = 0\}$ at the origin is closely related to its     link $\{(z_1,z_2) \in \C^2 \mid f(z_1,z_2) = 0, \abs{z_1}^2 + \abs{z_2}^2  = \varepsilon\}$, 
where $\varepsilon$ is sufficiently small. This link can be studied by looking at its splice diagram.  

Splicing itself is  originally a technique 
from link theory
  which constructs a new link out of two given links.  
We can also use splicing to decompose complicated links into easier links. 
This   decomposition procedure  has a nice description  for our case of plane curve singularities. Fix   
an embedded resolution of singularities $\pi: X\to \affinenspace{2}{\C}$ of $f^{-1}(0)$. The splice diagram $\Gamma$ associated to
$f$ and $\pi$ is then constructed
by taking the dual graph of the exceptional curves  of $\pi$, 
removing nodes  if necessary, where these nodes  correspond to exceptional curves,   and adding
decorations to the edges. 

The splicing of a splice diagram $\Gamma$ along an edge $e$ produces two new splice diagrams $\Gamma_R$ and $\Gamma_L$. 
As shown in Figure \ref{fig:splicingintro}, it divides $\Gamma$ into two pieces and then makes them again into a splice diagram
by adding   appropriate  multiplicities $M$ and  $M'$. Splicing the links of
$\Gamma_L$ and $\Gamma_R$ together, we obtain the link of $\Gamma$ (see \cite{EisenbudNeumann} for more details).

\makeatletter{}\begin{figure}
 \begin{tabular}{cc}
  \multicolumn{2}{c}{ \begin{subfigure}[b]{0.5\textwidth}
  \begin{center}
   \begin{tikzpicture}[scale=0.8]
	\renewcommand*{\locationOfNode}{2}
	\renewcommand*{\locationOfStart}{4}
	\renewcommand*{\heigthofedges}{1}
	\coordinate (dOne) at (-\locationOfStart,\heigthofedges);
	\coordinate (dN) at (-\locationOfStart,-\heigthofedges)  ;
	\coordinate (daOne) at (\locationOfStart,\heigthofedges)  ;
	\coordinate (daN) at (\locationOfStart,-\heigthofedges)  ;
	\coordinate (vL) at (-\locationOfNode,0) ;
	\coordinate (vR) at (\locationOfNode,0)  ;
	\draw (dOne)--(vL) node [midway, above ]  {$d_1$};
	\draw (dN)--(vL)  node [midway, above ]  {$d_n$};
	\draw[very thick] (vL)--(vR) node[very near start, above]{$d$} node[very near end, above]{$d'$};
	\draw (vR)--(daOne)  node [midway, above ]  {$d'_1$};
	\draw (vR)--(daN)  node [midway, above ]  {$d'_{n'}$};
	\draw node at ($ 1/2*(dOne) + 1/2*(dN) +(1/3,0) $) {\vdots};
	\draw node at ($ 1/2*(daOne) + 1/2*(daN) +(-1/3,0) $) {\vdots};
 	\draw node at ($(vL) + (0,-1) $){$v_L$};
 	\draw node at ($(vR) + (0,-1) $){$v_R$};
	\draw node at ($1/2*(vR) + 1/2*(vL) + (0,-1) $){$e$};
		\fill (vR) circle [radius=3pt];
	\fill (vL) circle [radius=3pt];
\end{tikzpicture}
  \end{center}
    \caption{The original diagram $\Gamma$}
    \label{fig:splicingexpl:Gamma:intro}
        \end{subfigure}} \\   
        \begin{subfigure}[b]{0.3\textwidth}
\begin{center}
 \begin{tikzpicture}[scale=0.8]
	\renewcommand*{\locationOfNode}{0}
	\renewcommand*{\locationOfStart}{2}
	\renewcommand*{\heigthofedges}{.75}
	\coordinate (sup) at (0,0);
	\coordinate (dOne) at (-\locationOfStart,\heigthofedges);
	\coordinate (dN) at (-\locationOfStart,-\heigthofedges)  ;
	\coordinate (daOne) at (\locationOfStart,0)  ; 
	\coordinate (vL) at (-\locationOfNode,0) ;
	\coordinate (vR) at (\locationOfNode,0)  ;
	\draw (dOne)--(vL) node [midway, above ]  {$d_1$};
	\draw (dN)--(vL)  node [midway, above ]  {$d_n$}; 
	\draw[->] ($(vL)+(sup)$)--($(daOne) + (sup)$)  node [ near start, above ]  {$d$} node [very near end, above] {$(M)$};
	\draw node at ($ 1/2*(dOne) + 1/2*(dN) +(1/3,0) $) {\vdots};
 	\fill (vL) circle [radius=3pt];
\end{tikzpicture}
               \caption{The diagram $\Gamma_L$}
                \label{fig:splicingexpl:GammaL:intro}
 \end{center}
         \end{subfigure}   &  
   
        \begin{subfigure}[b]{0.3\textwidth}
 \begin{center}
  \begin{tikzpicture}[scale=0.8]
\renewcommand*{\locationOfNode}{0}
	\renewcommand*{\locationOfStart}{2}
	\renewcommand*{\heigthofedges}{.75}
	\coordinate (sup) at (0,0);	
	\coordinate (dOne) at (-\locationOfStart,0);
	\coordinate (daOne) at (\locationOfStart,\heigthofedges)  ;
	\coordinate (daN) at (\locationOfStart,-\heigthofedges)  ;
	\coordinate (vL) at (-\locationOfNode,0) ;
	\coordinate (vR) at (\locationOfNode,0)  ;
	\draw[<-] ($(dOne)+(sup)$)--($(vR) + (sup)$)node [near end, above ]  {$d'$} node [very near start, above] {$(M')$};
			\draw (vR)--(daOne)  node [midway, above ]  {$d'_1$};
	\draw (vR)--(daN)  node [midway, above ]  {$d'_{n'}$}; 
	\draw node at ($ 1/2*(daOne) + 1/2*(daN) +(-1/3,0) $) {\vdots};
	\fill (vR) circle [radius=3pt];
\end{tikzpicture}
               \caption{The diagram $\Gamma_R$}
                \label{fig:splicingexpl:GammaR:intro}
 \end{center}
      \end{subfigure}     
        
        \end{tabular}

        \caption{Splicing a diagram $\Gamma$ along an edge $e$.} \label{fig:splicingintro} \end{figure}
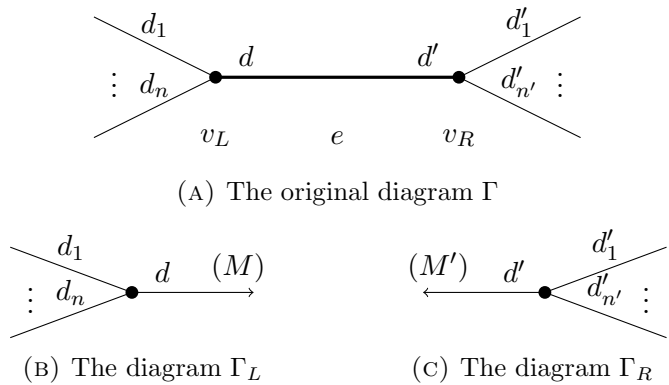   

 N\'emethi and Veys in \cite{NemethiVeys2012} applied this splicing  technique to the topological zeta function. 

\subsection*{Topological zeta functions}
 Denef and Loeser introduced the topological zeta function in \cite{DenefLoeser92} as follows: 
              let $E_j, j\in J$, be the irreducible components of $\pi^{-1}( f^{-1}(0) )$, $N_j$
the multiplicity of $\pi^*f$ along $E_j$ in  and $\nu_j-1 $ the multiplicity of $\pi^*(dx\wedge dy)$ along  $E_j$. The (local) topological zeta function is then 
\begin{equation}
  \topZeta{f}{s} = \sum_{\emptyset\neq I \subseteq J}\chitop(\EIO \cap \pi^{-1}(0)) \prod_{i\in I} 
 \frac{1}{N_is +\nu_i} \in \Q(s), \label{eq:definitiontopzetafintro}
\end{equation}
 where  $\EIO = \cap_{i\in I} E_i \setminus\left( \cup_{i \in J\setminus I} E_i \right)$. Using the existence of a minimal resolution, it is obvious that
that this is independent  of the chosen resolution. In general dimension the independence
 was originally proven by a limit argument using the $p$-adic zeta functions $\int_{p\Z_p^2} \abs{f}^s_p \abs{dx}$, where we 
 assumed that $f\in \Q[x,y]$. This can also be shown by using the  weak factorization theorem \cite{MR2013783}\cite{MR1896232} or 
 by considering the topological zeta function as a specialization of 
 the motivic zeta function, which we will discuss further on.

To obtain a splicing formula,  
 N\'emethi and Veys  incorporated a  differential form $\omega$ into the splice diagram and defined a topological zeta function $\topZeta{f,\omega}{s} = \topZeta{\Gamma}{s}$ for such 
splice diagrams (and thus with respect to $\omega$) by using \eqref{eq:definitiontopzetafintro}, where we redefine $\nu_i$ in terms of $\omega$.
This corresponds to considering the $p$-adic zeta functions $\int_{p\Z_p^2} \abs{f}^s_p \abs{\omega}$, given $f\in \Q[x,y]$. 
These zeta functions with respect to a differential form 
were  already introduced  in \cite{ACLM1}, \cite{ACLM2} and \cite{Veyszetafunctionsonsingularvarieties} with a restriction
on the support of $\omega$.  In \cite{VeysMEAndzetafunctionswithdifferential} and \cite{NemethiVeys2012} there is no restriction on the support
but only the topological zeta function is considered.

 \subsection*{Splicing formula and generalized monodromy conjecture}
 N\'emethi and Veys showed then that there is a  nice splicing formula connecting the involved diagrams in Figure 
 \ref{fig:splicingintro} and their topological zeta functions:
 \begin{equation}
 \topZeta{\Gamma}{s} = \topZeta{\Gamma_L}{s} + \topZeta{\Gamma_R}{s} - 
 \frac{1}{ (Ms + i)( M's +i' )} \label{eq:splicingformula:intro}.
\end{equation} These $i$ and $i'$ are also introduced when we splice $\Gamma$ and the involved zeta functions are with respect to some differential form, whose information is contained in 
the diagram.
  They used this to prove the generalized monodromy conjecture, which predicts the existence of a class of
   `allowed' differential 
  forms such that   the following holds:
   \begin{itemize}
  \item for every allowed form $\omega$ and every pole $s_0$ of $\topZeta{f,\omega}{s}$, $\exp(2\pi i s_0)$ is a local
  monodromy eigenvalue of $f$;
  \item $dx\wedge dy$ is allowed;
  \item every local monodromy eigenvalue of $f$ is obtained as a pole of $\topZeta{f,\omega}{s}$ for some allowed $\omega$.
 \end{itemize}

\subsection*{Motivic zeta functions} We consider here the  Grothendieck ring of varieties $\Grothendieckring{\C}$, its localization 
$\LocalizedGrothendieckring{\C}$ with respect to $\lefschetzmotive = \classof{\affinenspace{1}{\C}}$ and the completion $\CompletedGrothendieckRing{\C}$. 
We denote by $\classof{X}$ the class of the variety $X$. Also $\njets{n}{X}$ will be the scheme of $n$-jets of a variety $X$ and 
$\arcspace{X}$ will be the arc space of $X$. All of this  will be  introduced   in 
  more detail in Section \ref{section:motivicrings}.
As mentioned before, the topological zeta function can also  be considered as an 
avatar of the (local) motivic zeta function $\motivicZeta{f} \in \LocalizedGrothendieckring{\C}[[T]]$.  This motivic zeta function is defined by
 $
 \motivicZeta{f} := \sum_{n> 0} \classof{\motivicZetaFunctioncoef{n}} \lefschetzmotive^{-dn } T^n \in \LocalizedGrothendieckring{\C}[[T]]
$
where $ \motivicZetaFunctioncoef{n} := \{\varphi \in \njets{n}{\affinenspace{2}{\C}} \mid \ord_t f_n(\varphi) = n, \ \pi^n_0(\varphi) = 0\}$. 
There is also an explicit formula for  $\motivicZeta{f}$ in terms of a log resolution
similar to the one in \eqref{eq:definitiontopzetafintro}.
 
We introduce here a motivic zeta function with respect to a differential form. We do this by observing that 
$ \classof{\motivicZetaFunctioncoef{n}} \lefschetzmotive^{-dn }$
is the (naive) measure of a cylinder $\arcsoforder{n}$ in  the arc space $\arcspace{\affinenspace{2}{\C}}$. After defining 
a new motivic  measure $\mu_\omega$ with values in $\CompletedGrothendieckRing{\C}$, we obtain a motivic
zeta function   $\motivicZetaWithDifferentialForm{f}{\omega} := 
 \sum_{n > 0} \mu_\omega( \arcsoforder{n})  T^n \in \CompletedGrothendieckRing{\C}[[T]]$.
For this motivic zeta function $\motivicZetaWithDifferentialForm{f}{\omega}$ there exists again a formula:
\begin{equation}
  \motivicZetaWithDifferentialForm{f}{\omega} = \sum_{\emptyset\neq I \subset J} (\lefschetzmotive - 1)^{\lvert I\rvert }
 \classof{\EIO\cap \pi^{-1}(0)} \prod_{i\in I} \frac{T^{N_i}}{\lefschetzmotive^{v_i} - T^{N_i}} \in 
 \CompletedGrothendieckRing{\C}[[T]] \label{eq:defmotivic:intro} .
\end{equation}

We want to extend  the definition of the motivic zeta function to the case of splice diagrams, but it turns out to be a bit hard.
 To ease our
notation and to simplify our proofs, we introduce the notion of a {\it diagram}: this is actually the same as a splice diagram except we can choose how many nodes (of valency two)
we remove (instead of all of them). A diagram where no nodes were removed is called realizable and refining a diagram is adding nodes
again to the diagram to obtain a new diagram. The definition in \eqref{eq:defmotivic:intro} for
realizable diagrams then extends to all diagrams by 
choosing any refinement.

This leads us to our theorem.
\begin{theorem*}Consider a diagram  $\Gamma$  
and the splicing of   $\Gamma$ into $\Gamma_L$ and $\Gamma_R$. Then  we have
\[
\motivicZeta{\Gamma} = \motivicZeta{\Gamma_L} + \motivicZeta{\Gamma_R} - \frac{(\lefschetzmotive - 1)^2 T^{M + M'}}{ (\lefschetzmotive^i - T^M)(\lefschetzmotive^{i'} - T^{M'})}.
\]
\end{theorem*}

This formula specializes to \eqref{eq:splicingformula:intro}.

  \subsection*{Splice diagrams and monodromic motivic zeta functions} Consider $\muhat = \inverselimit_n\mu_n$, where 
  $\mu_n$ is the group of $n$-th roots of unity and the localized monodromic Grothendieck ring $\monodromicLocalizedGrothendieckring{\C}$. 
  This $\monodromicLocalizedGrothendieckring{\C}$ will be defined in more detail in Section \ref{section:motivicrings}.
The next step is to add a $\muhat $-action to the motivic zeta function, which is done by considering 
the monodromic motivic zeta function  
$\monodromicmotivicZetaFunction{f} := 
 \sum_{n > 0} \classof{\monodromicmotivicZetaFunctioncoef{n}} \lefschetzmotive^{-dn} T^n \in 
 \monodromicLocalizedGrothendieckring{\C}[[T]]$, where 
$ \monodromicmotivicZetaFunctioncoef{n} := \{\varphi \in \njets{n}{\affinenspace{2}{\C}} \mid   
 f_n(\varphi) \equiv  t^n \mod(t^{n+1}), \pi^n_0(\varphi) = 0\}$. The action of $\mu_n$
 on $\monodromicmotivicZetaFunctioncoef{n}$ is defined 
  by $a \cdot \varphi(t) = \varphi(at)$, where  $a\in \mu_n $ and $\varphi \in \njets{m}{\affinenspace{n}{\C}}$, 
  which induces an action of $\muhat$. 

  In section \ref{section:motivicrings} we define $\monodromicmotivicZetaFunction{f,\omega}$, which is 
  the monodromic motivic zeta function with respect to a differential form $\omega$.
  It turns out that
we cannot define $\monodromicmotivicZetaFunction{f,\omega}$ in terms of a (splice) diagram:  we will show that 
there exist  $\lambda, \lambda' \in \C\setminus\{0,1\}$ such that
 \[
  \monodromicmotivicZetaFunction{f_\lambda} \neq \monodromicmotivicZetaFunction{f_{\lambda'}},
 \] where $
  f_\lambda = xy^2(x-y)(x-\lambda y) \in \C[x,y]$. 
  For this $\lambda$ and $\lambda'$, we find
  that the associated splice diagrams are the same but  the monodromic motivic  zeta functions are not. Hence
  we cannot define  $\monodromicmotivicZetaFunction{f,\omega}$   in terms of a diagram. 
  
To construct such $\lambda$ and $\lambda'$ we use the Picard morphism constructed by Ekedahl in \cite{Ekedahl2009}.
  In the appendix we prove and explain the details of this construction  since
  \cite{Ekedahl2009} was never published.

 \subsection*{Generalized monodromy conjecture for motivic zeta functions} 
  The straightforward generalization to the motivic zeta functions turns out to be true. We will specify what we mean by a pole in Section \ref{section:mc}.
\begin{corollary*}Consider the set of allowed forms for a diagram $\Gamma$ of $f \in \C[x,y]$. It
satisfies the following conditions:
 \begin{itemize}
  \item for every allowed form $\omega$, every pole of $ \motivicZetaWithDifferentialForm{f}{\omega}$ induces 
  a monodromy eigenvalue. More specifically
  Theorem \ref{thm:generalizedmotivicmonodromyconjecture} holds.
  \item $dx\wedge dy$ is allowed;
  \item every monodromy eigenvalue is obtained as a pole of the motivic zeta function of $f$ with respect to $\omega$.
 \end{itemize}
\end{corollary*} 
However in the case of the monodromic motivic  zeta function we show that this does not hold.
First we give an easy example where an allowed form gives a non-desired pole of the twisted topological zeta function. The twisted topological zeta function will also  be introduced in Section \ref{section:motivicrings}.
This means that it must occur as a pole of the monodromic motivic zeta function and thus the generalized monodromy conjecture cannot
be valid.
Secondly we produce an example which shows that a subset of allowed forms is not sufficient: there exists no 
allowed differential form such that all the poles of the monodromic motivic zeta function induce monodromy eigenvalues and such that
a particular monodromy eigenvalue is obtained.

\vskip1cm
{\it Acknowledgements} I am very grateful to Wim Veys and Johannes Nicaise 
for their valuable suggestions. 

\makeatletter{}\section{Motivic zeta functions and Grothendieck rings\label{section:motivicrings}}
In this section we introduce the necessary Grothendieck rings and zeta functions. A variety will be
a complex algebraic variety, i.e. a reduced separated scheme of finite type over $\C$. The associated category will
be denoted by $\Var_\C$. Also fix a polynomial $f\in \C[x_1,\ldots, x_d]$, which we will also consider as
a morphism $\affinenspace{n}{\C} \to \affinenspace{1}{\C}$. See  \cite{MR1923998} and \cite{denefloeser2001}  
for more background information.

\subsection{Grothendieck rings}
The Grothendieck ring $\Grothendieckring{\C}$ of complex varieties is the abelian
group generated by isomorphism classes of varieties, where we denote the class of a variety $X$ by $\classof{X}$, 
subject to the relations
\begin{center}
  $\classof{X} = \classof{X \setminus Z} + \classof{Z}$, where $X$ is a variety and $Z$ is a closed subvariety of $X$.
\end{center} 

The ring structure is induced by defining $\classof{X} \cdot \classof{Y} = \classof{X \times Y} $
for varieties $X$ and $Y$.
We denote by $\lefschetzmotive$ the class of $\affinenspace{1}{\C}$ 
and the localization $\Grothendieckring{\C}[\lefschetzmotive^{-1}]$ is denoted by 
$\LocalizedGrothendieckring{\C}$.  Consider for every $i \in \N$ the subgroup $F^i$ of $\LocalizedGrothendieckring{\C}$ 
 generated by the elements $\quotientInGrothendieckRing{\classof{X}}{\lefschetzmotive^n}$,  
 where $X$ is a variety and $i \in \N$ such that $\dimension{X} - n \leq  -i$. These subgroups form a descending filtration on 
 $\LocalizedGrothendieckring{\C}$ and 
 its completion $\varprojlim_n (\LocalizedGrothendieckring{\C}\slash F^n)$  is denoted by $\CompletedGrothendieckRing{\C}$.

The group of $n$-th roots of unity is denoted by $\mu_n$. A {\em good} $\mu_n$-action 
 on a variety $X$ is an algebraic group action $\mu_n \times X \to X$
 such that each orbit is contained in an affine open subvariety.
  Consider   the group  $\muhat :=\varprojlim_n \mu_n$. A {\em good}
 $\muhat$-action on $X$ is  an action of $\muhat$ on $X$ which factors through a good  $\mu_n$-action for 
 some $n\in \N$.

 Call   two $\muhat$-actions on varieties $X$ and $Y$  isomorphic if there is an isomorphism
 between $X$ and $Y$ which also preserves  the $\muhat$-action. 
 The monodromic Grothendieck ring  of complex varieties with good $\muhat$-action $\monodromicGrothendieckRing{\C}$ is 
 defined as the abelian group generated by isomorphism classes of varieties with a  good $\muhat$-action,
 where the class of a  variety $X$ with action $\alpha:\muhat \times X \to X$ is denoted by $\classof{X,\alpha}$ or $\classof{X}$, subject to the 
 relations
 \begin{itemize}
  \item $\classof{X} = \classof{X\setminus Z} + \classof{Z}$, where $X$ is a variety with $\muhat$-action, $Z$ is a closed subvariety invariant
   under the action, and the actions on $X\setminus Z$ and $Z$ are induced by the one on $X$;
   \item 
$\classof{ V }= \classof{X \times \affinenspace{n}{\C}}$, where $V \to X$ is vector bundle of rank $n$ over
a variety $X$ with
 a $\muhat$-action which is linear over the action on $X$. (We are following  \cite{MR2106970} here)
 \end{itemize}
 The ring structure can be defined in the same way as before. 
 Again we denote by $\lefschetzmotive$ the class of $\affinenspace{1}{\C}$, where we equip it with the trivial $\muhat$-action
 and we define $\monodromicLocalizedGrothendieckring{\C}$ to
 be the localization $\monodromicGrothendieckRing{\C}[\lefschetzmotive^{-1}]$.
 Analogously for every $i\in \N$ we define $\hat{F}^i$ as  the subgroup generated by the elements $\quotientInGrothendieckRing{\classof{X}}{\lefschetzmotive^n}$  
 where $X$ is a variety such that $\dimension{X} - n \leq  -i$. Its completion 
  $\varprojlim_n (\monodromicLocalizedGrothendieckring{\C}\slash \hat{F}^n)$  is the monodromic completed Grothendieck ring $ \monodromicCompletedGrothendieckRing{\C}$.

\subsection{$n$-jets and  arcs } Fix a variety $X$ of dimension $d$ in this subsection and let $n\in \N$.
 Recall that the functor 
 \[
  \Sch_\C \to \Set : Y \mapsto X\left(Y\times_\C \Spec\left( \frac{\C[t]}{(t^{n+1})}\right) \right) = \Hom_\C \left(Y\times_\C \Spec\left(  \frac{\C[t]}{(t^{n+1})}\right),X \right)
 \]
is representable by a scheme  $\njets{n}{X}$. 
In particular we have $\njets{n}{X} (F) = X\left(\frac{F[t]}{(t^{n+1})}\right)$ for any field extension $F$ of $\C$. 
This scheme is called the scheme of $n$-jets of $X$. Remark also that $\njets{0}{X} = X$.

The truncation morphisms $\pi_m^n : \njets{n}{X} \to \njets{m}{X}$ are affine, where $n \geq m$,
and thus
we can consider the scheme $\arcspace{X} = \varprojlim_n \njets{n}{X}$, which is called the arc space of $X$. It is equipped
with projection morphisms $\pi_n : \arcspace{X} \to \njets{n}{X}$.
A morphism $f: X\to Y$ of algebraic varieties induces
morphisms $f_n : \njets{n}{X} \to \njets{n}{Y}$ and $f: \arcspace{X} \to \arcspace{Y}$, 
which are compatible with the projection
morphisms.

\subsection{Motivic zeta function} We  consider
\[
 \motivicZetaFunctioncoef{n} := \{\varphi \in \njets{n}{\affinenspace{d}{\C}} \mid \ord_t f_n(\varphi) = n, \ \pi^n_0(\varphi) = 0\}
\] for $n \in \N$. This is a locally closed subset of $\njets{n}{\affinenspace{d}{\C}} $ and thus defines 
a class $\classof{\motivicZetaFunctioncoef{n} }$ in $\Grothendieckring{\C}$.
The (local) motivic zeta function of $f$ (at the origin 0) is then defined as
\[
 \motivicZeta{f} := \sum_{n> 0} \classof{\motivicZetaFunctioncoef{n}} \lefschetzmotive^{-dn } T^n \in \LocalizedGrothendieckring{\C}[[T]].
\]

Our goal now is to incorporate a (regular) differential  form $\omega$ of maximal degree in this definition. 
This is done in \cite{ACLM1}, \cite{ACLM2} and \cite{Veyszetafunctionsonsingularvarieties} with a restriction
on the support of $\omega$. In \cite{VeysMEAndzetafunctionswithdifferential} and \cite{NemethiVeys2012} there is no restriction on the support
but only the topological zeta function is used.
We will show how we can see the coefficients of the motivic zeta function as the motivic measure of a 
subset of $\arcspace{\affinenspace{d}{\C}}$.  Assume that $X$ is a smooth variety of dimension $d$.
Let $C$ be the boolean algebra of cylindrical subsets of $\arcspace{X}$, 
i.e. subsets of the form $\pi_n^{-1}(A')$, where $A'$ is
a constructible subset  of $\njets{n}{X}$ for some $n\in \N$. 
Recall that for cylindrical subsets $A $ we have
that $\pi_n(A) \lefschetzmotive^{-dn} \in \LocalizedGrothendieckring{\C}$ 
stabilizes if $n$ tends to infinity;   we denote this element by $\mu(A)$. This is called the
(naive) motivic measure of $A$. Define $ \arcsoforder{n}$ to be
$ \{ \varphi \in \arcspace{\affinenspace{d}{\C}} \mid \ord_t f(\varphi) = n, \pi_0(\varphi) = 0\}$
and note that  $\mu( \arcsoforder{n}) = \classof{\motivicZetaFunctioncoef{n}} \lefschetzmotive^{-dn}$. Hence
the coefficients of the  motivic zeta function are the measures of $\arcsoforder{n}$.

Consider now a regular differential form $\omega$  of maximal degree on $X$. Remark that $\divisor(\omega)$
is a divisor and consider 
$\arcsoforderagainstdifform{e} = \{ \varphi \in \arcspace{X} \mid \ord_t \divisor(\omega)(\varphi) = e\}$, which
is a cylindrical subset. We define
\[
 \mu_\omega : C \to \CompletedGrothendieckRing{\C} : A \mapsto \sum_{e\in \mathbb{N} } \mu(A \cap \Delta_e)
 \lefschetzmotive^{-e}
\]
and we call this the motivic measure with respect to $\omega$. This sum converges in $\CompletedGrothendieckRing{\C}$ since 
there exists an $i\in \N$ such that 
$ \mu(A \cap \Delta_e)\lefschetzmotive^{-i} \in F^0$  for all $e \in \N$.

 \begin{definition} The (local)  motivic zeta function of $f$ with respect to $\omega$ (at 0) is
  \[
 \motivicZetaWithDifferentialForm{f}{\omega} := 
 \sum_{i > 0} \mu_\omega( \arcsoforder{n})  T^i \in \CompletedGrothendieckRing{\C}[[T]].\]
\end{definition}
We find that  $\motivicZetaWithDifferentialForm{f}{\omega}$ coincides with $\motivicZeta{f}$ if
$\omega $ is the standard form $dx_1 \wedge \cdots \wedge dx_n$ and if 
we consider the coefficients in $\CompletedGrothendieckRing{\C}$. 

We also have a formula in terms of an embedded resolution. Let $\pi: X \to \affinenspace{n}{\C}$ be an embedded resolution
of singularities of $f$ and $\omega$, i.e. a proper morphism such that
$\pi^{-1}(f^{-1}(0) \cup \supp{\omega})$ is a strict normal crossing divisor and such that is an isomorphism outside 
$  f^{-1}(0) \cup \supp{\omega}$.
Let $E_j, j\in J$, be the irreducible components of $\pi^{-1}(f^{-1}(0)\cup \supp{\omega})$, $N_j$
the multiplicity of $E_j$ in $\pi^*f$ and $\nu_j-1 $ the multiplicity of $\pi^*\omega$ along $E_j$. Remark that $(N_j,\nu_j) \neq (0,0)$
for all $j\in J$, but $N_j=0$ is possible for some $j\in J$ since it can happen that $\supp(\omega)\not \subset f^{-1}(0)$. 

\begin{theorem}\label{theorem:denefformulamotiviczetafunction}

Define $\EI = \cap_{i\in I} E_i$ and $\EIO = \cap_{i\in I} E_i \setminus\left( \cup_{i \in J\setminus I} E_i \right)$ for $I \subset J$.
Then we have
 \[
 \motivicZetaWithDifferentialForm{f}{\omega} = \sum_{\emptyset\neq I \subset J} (\lefschetzmotive - 1)^{\lvert I\rvert }
 \classof{\EIO\cap \pi^{-1}(0)} \prod_{i\in I} \frac{T^{N_i}}{\lefschetzmotive^{\nu_i} - T^{N_i}} \in 
 \CompletedGrothendieckRing{\C}[[T]]
 \]
\end{theorem}
\begin{proof}
Using the techniques of \cite[Theorem 2.4]{MR1923998} one can show this easily.
\end{proof}

\subsection{Monodromic motivic zeta function}
 Consider \[
 \monodromicmotivicZetaFunctioncoef{n} := \{\varphi \in \njets{n}{\affinenspace{d}{\C}} \mid   
 f_n(\varphi) \equiv  t^n \mod\ (t^{n+1}), \pi^n_0(\varphi) = 0\}.
\] for $n \in \N$. This is a closed subset of $ \njets{n}{\affinenspace{d}{\C}}$.
We have a (natural) $\mu_n$-action defined by
$a \cdot \varphi(t) = \varphi(at)$, where $a\in \mu_n $ and $\varphi \in \njets{n}{\affinenspace{d}{\C}}$. Hence
 we   can consider $\classof{\monodromicmotivicZetaFunctioncoef{m}}$ in  $\monodromicLocalizedGrothendieckring{\C}$. The 
 monodromic motivic zeta function of $f$ is then
\begin{equation}\label{eq:monodromicmotiviczetafunctionwithoutdiff}
 \monodromicmotivicZetaFunction{f} := 
 \sum_{n > 0} \classof{\monodromicmotivicZetaFunctioncoef{n}} \lefschetzmotive^{-dn} T^n \in 
 \monodromicLocalizedGrothendieckring{\C}[[T]].
\end{equation}

Look   at   $\arcsoforderandac{n} = \{\varphi \in \arcspace{\affinenspace{d}{\C}} \mid f(\varphi) \equiv  t^n \mod(t^{n+1}), \pi_0(\varphi) = 0 \}$. This
has an action of $\mu_n$ like before.
Then $\arcsoforderandac{n}\cap \Delta_e$ is a cylindrical subset for every $e\in \mathbb{N}$
and the action induces an action on $\pi_m( \arcsoforderandac{n}\cap \Delta_e)$ for every $m\in \N$. 
The sequence $\pi_m(\arcsoforder{n}\cap \Delta_e)\lefschetzmotive^{-md}$
stabilizes in $\monodromicCompletedGrothendieckRing{\C}$ if $m$ tends to $\infty$ and we denote this element
by $\mu( \arcsoforderandac{n}\cap \Delta_e)$. 
This leads us to 
\[
 \mu_\omega (\arcsoforderandac{n}) = \sum_{e\in \mathbb{N} } \mu(\arcsoforderandac{n} \cap \Delta_e)
 \lefschetzmotive^{-e} \in \monodromicCompletedGrothendieckRing{\C}.
\] This sum converges in $\monodromicCompletedGrothendieckRing{\C}$ by the same argument as before.

\begin{definition}
 The (local) monodromic motivic zeta function of $f$ with respect to $\omega$ (at the origin) is  
\[
 \monodromicmotivicZetaFunctionT{f, \omega}{T} := \sum_{i> 0}\mu_\omega (\arcsoforderandac{i}) T^i \in  \monodromicCompletedGrothendieckRing{\C} [[T]].
\]

\end{definition}
Remark that this definition coincides again with the one in (\ref{eq:monodromicmotiviczetafunctionwithoutdiff}) 
if $\omega $ is the standard form and we consider the coefficients in $ \monodromicCompletedGrothendieckRing{\C}$.

We also have a formula for this zeta function. Consider again the situation of Theorem \ref{theorem:denefformulamotiviczetafunction}.
 Suppose  $\emptyset \neq I\subseteq J$. Define $m_I = \gcd_{i \in I}(N_i)$. We will 
 introduce $\EIOT$ as a unramified Galois cover of $\EIO$ with Galois group 
 $\mu_{m_I}$. Let $U$ be an affine Zariski open such that $f\circ h = u v^{m_I}$,
 where $u$ is a unit and $v$ a regular function on  $U$. Then the restriction
 of $\EIOT$ above $\EIO\cap U$ is
 defined as
 \[
  \{(z,y) \in \affinenspace{1}{\C} \times U \mid z^{m_I} = u^{-1}\}.
 \]
 Since another choice of $u$ and $v$ induces an isomorphism, 
 these covers glue to a 
 finite Galois cover $\EIOT$ of  $\EIO$. The (natural) $\mu_{m_I}$-action is obtained 
 by multiplying the $z$-coordinates with elements of $\mu_{m_I}$, which gives us an element  
 $\classof{\EIOT}$ in  $\monodromicGrothendieckRing{\C}$.

 \begin{theorem}\label{theorem:denefformulamonodromicmotiviczetafunction}
  We have the following equality:
 \[
 \motivicZetaWithDifferentialForm{f}{\omega} = \sum_{\emptyset\neq I \subset J} (\lefschetzmotive - 1)^{\lvert I\rvert -1}
 \classof{\EIOT\cap \pi^{-1}(0)} \prod_{i\in I} \frac{T^{N_i}}{\lefschetzmotive^{\nu_i} - T^{N_i}}
 \in \monodromicCompletedGrothendieckRing{\C}[[T]].
 \]
 \end{theorem}
\begin{proof}
Using the techniques of \cite[Theorem 2.4]{MR1923998} one can show this easily.
\end{proof}

 \subsection{Topological and other zeta functions}
 Given a ring morphism $\chi: \CompletedGrothendieckRing{\C} \to R$, we can consider the specialization of these motivic zeta functions, i.e.
\[ 
\motivicZetaWithDifferentialForm{f}{\omega} = \sum_{i > 0} \chi(\mu_\omega( \arcsoforder{n}))  T^i \in R[[T]].
\]
 In this sense we can obtain several other and already known zeta  functions such as $ p$-adic zeta functions and Hodge zeta functions.

\subsubsection{Topological zeta function}
The topological zeta function will be used and thus we discuss this incarnation in more detail. 
The topological Euler characteristic $\chitop(X) \in \Z$ of a variety $X$ has the following properties:
\begin{itemize}
 \item $\chitop(X) = \chitop(X\setminus Z) + \chitop(Z) $ for all varieties $X$ and closed subvarieties $Z$ of $X$,
 \item $\chitop(X\times Y) = \chitop(X) \cdot \chitop(Y)$ for all varieties $X$ and $Y$,
 \item $\chitop(\affinenspace{1}{\C}) = 1$
\end{itemize}
This implies that we can consider it as a ring morphism $\chitop : \LocalizedGrothendieckring{\C} \to \Z$. Since 
$\chitop(\lefschetzmotive) = 1$, we cannot extend it to a morphism from $\CompletedGrothendieckRing{\C}$. As discussed in \cite{denefloeser2001}, 
we can still apply $\chitop $ to elements of $\CompletedGrothendieckRing{\C}$ 
which are the image of an element of $\LocalizedGrothendieckring{\C}$.

The local topological zeta function
of $f$ is defined as the rational function
\[
 \topZeta{f,\omega}{s} = \sum_{\emptyset\neq I \subseteq J}\chitop(\EIO\cap \pi^{-1}(0)) \prod_{i\in I} 
 \frac{1}{N_is +\nu_i} \in \Q(s),
\] 
and the twisted local topological zeta function  as
\[
 \topZetaTwisted{f}{\omega}{e}{s} = 
 \sum_{\stackrel{\emptyset\neq I \subseteq J,}{ e\vert m_I}}\chitop(\EIO \cap \pi^{-1}(0)) 
 \prod_{i\in I} \frac{1}{N_i s+ \nu_i} \in \Q(s),
\] where $e\in \N$.

Given a character $\alpha$ of $\muhat$, there is a natural ring homomorphism 
\[
 \chitop(\cdot, \alpha) : \monodromicLocalizedGrothendieckring{\C} \to \Z : X \mapsto \sum_{q\geq 0} \dim H^q(X,\C)_\alpha
\]
where $H^*(X,\C)_\alpha$ is the part of $H^*(X,\C)$ on which $\muhat $ acts by multiplication by $\alpha$.

Remark that there always exists a 
 character $\alpha$
of given order $e$ and that  $\chitop(X,\alpha)$ only depends on $\alpha$. We denote this   by $\chitop^{(e)}(X) = \chitop(X,\alpha) $. 
We can apply $\chitop$   to 
$ \motivicZetaWithDifferentialFormT{f}{\omega}{\lefschetzmotive^{-n}}$ where $n\in \N$  
since it is equal to \[
\sum_{\emptyset\neq I \subset J}  
 \classof{\EIO\cap \pi^{-1}(0)} \prod_{i\in I} 
 \frac{ (\lefschetzmotive - 1)\lefschetzmotive^{-nN_i}}{\lefschetzmotive^{\nu_i} - \lefschetzmotive^{-nN_i}} 
 = \sum_{\emptyset\neq I \subset J} (\lefschetzmotive - 1)^{\lvert I\rvert }
 \classof{\EIO\cap \pi^{-1}(0)} \prod_{i\in I} \frac{1}{\classof{\projectivenspace{ N_in + \nu_i -1}{\C}}}. 
 \]
Another definition of  $\topZeta{f,\omega}{s}$ is then the unique rational function such that 
  \[ \topZeta{f,\omega}{n} = \chitop( \motivicZetaWithDifferentialFormT{f}{\omega}{\lefschetzmotive^{-n}})\]
  for all $n\in \N$. Analogous we have that $\topZetaTwisted{f}{\omega}{e}{s}$ is the unique rational function such that 
  \[   
                            \topZetaTwisted{f}{\omega}{e}{n} = \chitop^{(e)}((\lefschetzmotive-1)\monodromicmotivicZetaFunctionT{f,\omega}{\lefschetzmotive^{-n	}}).
                           \]  for all $n\in \N$. In this sense the (twisted) topological zeta functions are avatars of 
                           the motivic zeta functions.

\makeatletter{}\section{Splice diagrams} 
This section is dedicated to the notion of splice diagrams as described in \cite{NemethiVeys2012}. We will not discuss it
in full generality but rather stick to the case of plane curve singularities. 
Consider a polynomial $f \in \C[x,y]$, a  (regular) differential $2$-form $\omega$ on $\affinenspace{2}{\C}$
and an embedded resolution of singularities $\pi : X \to \affinenspace{2}{\C}$ for $(f,\omega)$. 

The topology of the singularity  can be described by means of the dual graph
$G = G_\pi $ associated to $\pi$, $f$ and $\omega$. 
Let $E_i, i\in J$, be the 
irreducible components of  $\pi^{-1}(  f^{-1}(0) \cup \supp  \omega   )$. We   have then three types of components: exceptional curves,
strict transforms of components of $\{f=0\}$ and strict transforms of components of $\supp(\omega)$. This type is not unique, i.e.
a component can be a strict transform of a component of $\{f=0\}$ and a strict transform of a component of $\supp(\omega)$, 
but this is the only case where a component can have multiple types.

Each exceptional curve $E_i$ determines a vertex of $G$ and edges correspond to the intersection points of the
exceptional curves. Note that this $G$ is a tree, each exceptional curve is rational and $\det (-I(G)) = 1$, where
$I(H)$ is the negative definite intersection matrix $(E_i\cdot E_j)_{i,j\in H}$ if $H$ is a subset of the nodes of $G$.
These are exactly the conditions of an integral homology sphere and thus we can use all the machinery developed in 
\cite{NemethiVeys2012}.

We can now talk about the splice diagram $\Gamma = \Gamma_\pi(f,\omega)$. The underlying graph is 
the one obtained by removing all nodes of $G$ of valency $2$. We add to this graph some decorations. On each pair $(v,e)$
where $v $  is a node of 
$\Gamma$ and  $e$ is an edge starting at $v$, we have the decoration $d_{ve} = \det(-I(G_{ve}))$, where $G_{ve}$ 
is the component of $G\setminus\{v\}$ in the direction of $e$.

Each irreducible component of the strict transform of  $\{f=0\}$ intersecting in the exceptional component $E$
corresponding to the node  $v$ is represented by
an arrow $a$ attached at $v$ and has the multiplicity $N_a$ of $f$ along $E$ as a decoration. 
Similarly 
a   component of a strict transform of   $\supp \omega $ is being displayed by a dotted arrow $a$
and again the multiplicity $\nu_a-1$ is the associated decoration.

It is important to stress that we are only dealing here with plane curve singularities. 
All these can easily be extended to the case 
of  integral homology spheres as in \cite{NemethiVeys2012}.

\begin{example}
 Consider the cusp $f=x^3 - y^2$ , its minimal embedded resolution $\pi$ and differential form 
 $\omega= x^4y^5dx\wedge dy$. Then its splice diagram is the following:
 \begin{center}
 \makeatletter{}
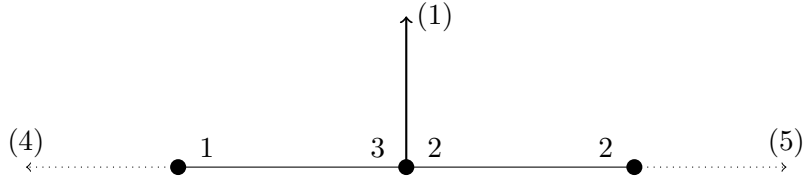
\begin{tikzpicture}
	\newcommand*{\locationOfLeft}{3}
	\newcommand*{\heightOfArrow}{2}
	\renewcommand*{\locationOfStart}{4}
	\renewcommand*{\heigthofedges}{1}
	\coordinate (een) at (-\locationOfLeft, 0);
	\coordinate (twee) at (0, 0);
	\coordinate (drie) at (\locationOfLeft, 0);
	\coordinate (vier) at (0, \heightOfArrow);
	\coordinate (vijf) at ( -\locationOfLeft-\heightOfArrow,0);
	\coordinate (zes) at ( +\locationOfLeft+\heightOfArrow,0);
	\draw (een)--(twee) node[very near start, above]{$1$} node[very near end, above]{$3$};
	\draw (twee)--(drie) node[very near start, above]{$2$} node[very near end, above]{$2$};
	\fill (een) circle [radius=3pt];
	\fill (twee) circle [radius=3pt];
	\fill (drie) circle [radius=3pt];
	\draw[thick,->] (twee)--(vier);
	\draw node[right] at (vier) {(1)};
	\draw node[above] at (vijf) {(4)};
	\draw node[above] at (zes) {(5)};
	\draw[dotted, ->] (een)--(vijf);
        \draw[dotted, ->] (drie)--(zes);
	
\end{tikzpicture} 
 \end{center}
\end{example}
We have the following properties of the decorations $\{d_{ve}\}_{e,v}$ of $\Gamma$:  $d_{ve} \geq 1$,
$\{d_{ve}\}_e$ are pairwise coprime for a fixed node $v$ and any edge determinant $q_e$ is positive. 
Recall that this edge determinant is $\det(-I(G_e))$, where $G_e$ are the exceptional curves  of $G$ that lie
on $e$. There is an easier formula. Consider the  situation in Figure 
\ref{fig:edgesituation}.
The edge determinant $q_e$ is then equal to 
$
 dd' - DD'
$  where $D = \prod_{i=1}^nd_i$ and $D'=\prod_{i=1}^{n'}d'_i$.

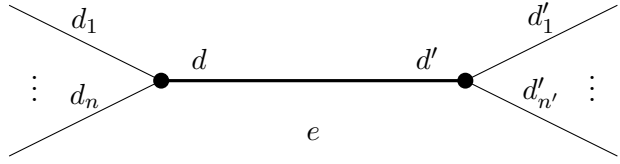
\begin{figure}\begin{center}
\makeatletter{}\begin{tikzpicture}
	\renewcommand*{\locationOfNode}{2}
	\renewcommand*{\locationOfStart}{4}
	\renewcommand*{\heigthofedges}{1}
	\coordinate (dOne) at (-\locationOfStart,\heigthofedges);
	\coordinate (dN) at (-\locationOfStart,-\heigthofedges)  ;
	\coordinate (daOne) at (\locationOfStart,\heigthofedges)  ;
	\coordinate (daN) at (\locationOfStart,-\heigthofedges)  ;
	\coordinate (vL) at (-\locationOfNode,0) ;
	\coordinate (vR) at (\locationOfNode,0)  ;
	\draw (dOne)--(vL) node [midway, above ]  {$d_1$};
	\draw (dN)--(vL)  node [midway, above ]  {$d_n$};
	\draw[very thick] (vL)--(vR) node[very near start, above]{$d$} node[very near end, above]{$d'$};
	\draw (vR)--(daOne)  node [midway, above ]  {$d'_1$};
	\draw (vR)--(daN)  node [midway, above ]  {$d'_{n'}$};
	\draw node at ($ 1/2*(dOne) + 1/2*(dN) +(1/3,0) $) {\vdots};
	\draw node at ($ 1/2*(daOne) + 1/2*(daN) +(-1/3,0) $) {\vdots};
	\draw node at ($1/2*(vR) + 1/2*(vL) + (0,-.7) $){$e$};
		\fill (vR) circle [radius=3pt];
	\fill (vL) circle [radius=3pt];
\end{tikzpicture} \caption{\label{fig:edgesituation}Generic situation}
\end{center}
\end{figure}

It turns out that this splice diagram is very useful for computing the multiplicities of
all the exceptional components. We only need 
the multiplicities at the strict transforms and the decorations on the splice diagram.
 For a node $v$ of $\Gamma$, which corresponds to an exceptional curve,  we have \begin{equation}\label{eq:formulafornodes}
N_v = \sum_{a \text{ arrow}} N_al_{va }                       
                       \end{equation}
where $l_{va}$ is the product of the edge decorations adjacent to the path from $v$ to $a$ but not on it.
Analogously, we have 
\begin{equation}\label{eq:formulafordifferentialform}
 \nu_v = \sum_{w \text{ node}} (2-\delta_w)l_{vw} + \sum_{a \text{ dotted arrow }} l_{va}(\nu_a-1).
\end{equation}
 This $\delta_v$ is the valency of the node $v$ considered without arrows.

 \subsection{Diagrams}
 We introduce now what we will call a diagram. While creating
 splice diagrams, we are deleting nodes of valency two in the graph. But the formulas still work without 
 doing so. Consider now dual graphs of resolutions as before, were we may choose how many valency two nodes we have removed.
 The decorations are still the same.
  From now on we will call this a  {\it diagram}.
  \begin{definition}
We call a diagram $\Gamma$ {\it realizable} if it is a dual graph of a resolution of singularities.  
A diagram with no nodes of valency 2 is called {\it minimally reduced}.
\end{definition}

So a splice diagram is a minimally reduced diagram.  
Of course every diagram can be reduced to a  minimally reduced  diagram by deleting all the nodes of valency two.

\begin{lemma}\begin{enumerate}
              \item If a diagram satisfies  $q_e = 1$ for all edges $e$ and $d_{va}=1$  for all (dotted) arrows $a$ attached at a node $v$,
then it must be realizable. 
\item Every reduced diagram can be `extended' or `refined' to a realizable diagram. This can be done by 
 (re)adding nodes (of valency 2) on those edges with $q_e > 1$ and 
arrows with $d_{va} >1$. 
\item The possible ways to add nodes to the edge $e$ is not unique and corresponds to smooth refinements of the fan associated to the cone 
$(D,d) \mathbb{R}_{\geq 0} + (d',D')\mathbb{R}_{\geq 0} \subseteq \mathbb{R}^2$ where   $D = \prod_{i=1 }^n d_i$ and
$D'= \prod_{i=1}^{n'} d'_i$, where we consider the situation as in Figure \ref{fig:edgesituation}.

             \end{enumerate}
\end{lemma}
 \begin{proof}
 This follows from interpreting toric concepts   described in \cite[Chapter 10]{MR2810322}.
 \end{proof}

\subsection{Zeta functions of diagrams}
Recall that we have a formula for our zeta functions in terms of an embedded resolution. In \cite{NemethiVeys2012}
N\'emethi and Veys define a topological zeta function in terms of the splice diagram. 
We will give here another definition which is easily seen to be equivalent and is actually just the formula as in
Theorem \ref{theorem:denefformulamotiviczetafunction}.

Let $\Gamma$ be a realizable diagram. We   define the {\it topological zeta function} of $\Gamma$ to be
\begin{eqnarray*}
 \topZeta{\Gamma}{s}  &:=&
 \sum_{v \text{ is a  node}} \frac{2-\delta_v}{ ( N_vs + \nu_v )} \\
  & & + \sum_{e= (v,w) \text{ is an edge}} 
  \frac{1}{( N_vs + \nu_v)({N_w}s + {\nu_w})} \\
 & & + \sum_{a \text{ (dotted) arrow at $v$}} 
 \frac{1}{( {N_v}s + {\nu_v})( {N_a}s + {\nu_a} )} 
 \in \Q(s).
\end{eqnarray*}
If $\Gamma$ is not realizable, let $\Gamma' $ be a  realizable refinement of $\Gamma$ and define
$ \topZeta{\Gamma}{s} :=  \topZeta{\Gamma'}{s}$. To prove that this is well-defined, we remark two things:
\begin{itemize}
\item If you add a node on a realizable diagram $\Gamma$ such that the resulting diagram is still 
realizable, the topological zeta function does not change. 
This coincides with blowing up in an intersection point of the two exceptional curves corresponding to the adjoining nodes. 
 
\item For any diagram, you can go from a realizable refinement to any other realizable refinement by blowing up points and doing 
the opposite operation.

 \end{itemize} In what follows we will only give the definition of the considered zeta functions in the case of a realizable diagram since
 these remarks will also imply that it is well-defined in  those cases. 
 \makeatletter{}\begin{figure}
 \begin{tabular}{cc}
  \multicolumn{2}{c}{ \begin{subfigure}[b]{0.5\textwidth}
\makeatletter{}\begin{tikzpicture}[scale=0.8]
	\renewcommand*{\locationOfNode}{2}
	\renewcommand*{\locationOfStart}{4}
	\renewcommand*{\heigthofedges}{1}
	\coordinate (dOne) at (-\locationOfStart,\heigthofedges);
	\coordinate (dN) at (-\locationOfStart,-\heigthofedges)  ;
	\coordinate (daOne) at (\locationOfStart,\heigthofedges)  ;
	\coordinate (daN) at (\locationOfStart,-\heigthofedges)  ;
	\coordinate (vL) at (-\locationOfNode,0) ;
	\coordinate (vR) at (\locationOfNode,0)  ;
	\draw (dOne)--(vL) node [midway, above ]  {$d_1$};
	\draw (dN)--(vL)  node [midway, above ]  {$d_n$};
	\draw[very thick] (vL)--(vR) node[very near start, above]{$d$} node[very near end, above]{$d'$};
	\draw (vR)--(daOne)  node [midway, above ]  {$d'_1$};
	\draw (vR)--(daN)  node [midway, above ]  {$d'_{n'}$};
	\draw node at ($ 1/2*(dOne) + 1/2*(dN) +(1/3,0) $) {\vdots};
	\draw node at ($ 1/2*(daOne) + 1/2*(daN) +(-1/3,0) $) {\vdots};
 	\draw node at ($(vL) + (0,-1) $){$v_L$};
 	\draw node at ($(vR) + (0,-1) $){$v_R$};
	\draw node at ($1/2*(vR) + 1/2*(vL) + (0,-1) $){$e$};
		\fill (vR) circle [radius=3pt];
	\fill (vL) circle [radius=3pt];
\end{tikzpicture} 
    \caption{The original diagram $\Gamma$}
    \label{fig:splicingexpl:Gamma}
        \end{subfigure}} \\   
        \begin{subfigure}[b]{0.3\textwidth}
 
\begin{tikzpicture}
	\renewcommand*{\locationOfNode}{1}
	\renewcommand*{\locationOfStart}{2}
	\renewcommand*{\heigthofedges}{.5}
	\coordinate (sup) at (0,.04);
	\coordinate (dOne) at (-\locationOfStart,\heigthofedges);
	\coordinate (dN) at (-\locationOfStart,-\heigthofedges)  ;
	\coordinate (daOne) at (\locationOfStart,0)  ; 
	\coordinate (vL) at (-\locationOfNode,0) ;
	\coordinate (vR) at (\locationOfNode,0)  ;
	\draw (dOne)--(vL) node [midway, above ]  {$d_1$};
	\draw (dN)--(vL)  node [midway, above ]  {$d_n$};
	\draw  (vL)--(vR) node[very near start, above]{$d$} node[very near end, above]{$d'$};
	\draw[->] ($(vR)+(sup)$)--($(daOne) + (sup)$)  node [ near start, above ]  {$D'$} node [very near end, above] {$(M)$};
	\draw[->, dotted] ($(vR)-(sup)$)--($(daOne) - (sup)$)   node [very near end,below] {$(i-1)$};
	\draw node at ($ 1/2*(dOne) + 1/2*(dN) +(1/3,0) $) {\vdots};
		\fill (vR) circle [radius=3pt];
	\fill (vL) circle [radius=3pt];
\end{tikzpicture}

                \caption{The diagram $\Gamma_L$}
                \label{fig:splicingexpl:GammaL}
        \end{subfigure}   &  
   
        \begin{subfigure}[b]{0.3\textwidth}
 
\begin{tikzpicture}
\renewcommand*{\locationOfNode}{1}
	\renewcommand*{\locationOfStart}{2}
	\renewcommand*{\heigthofedges}{.5}
	\coordinate (sup) at (0,.04);	
	\coordinate (dOne) at (-\locationOfStart,0);
	\coordinate (daOne) at (\locationOfStart,\heigthofedges)  ;
	\coordinate (daN) at (\locationOfStart,-\heigthofedges)  ;
	\coordinate (vL) at (-\locationOfNode,0) ;
	\coordinate (vR) at (\locationOfNode,0)  ;
	\draw[<-] ($(dOne)+(sup)$)--($(vL) + (sup)$)node [near end, above ]  {$D$} node [very near start, above] {$(M')$};
	\draw[<-, dotted] ($(dOne)-(sup)$)--($(vL) - (sup)$)  node [very near start, below	] {$(i'-1)$};
		 	\draw  (vL)--(vR) node[very near start, above]{$d$} node[very near end, above]{$d'$};
	\draw (vR)--(daOne)  node [midway, above ]  {$d'_1$};
	\draw (vR)--(daN)  node [midway, above ]  {$d'_{n'}$}; 
	\draw node at ($ 1/2*(daOne) + 1/2*(daN) +(-1/3,0) $) {\vdots};
		\fill (vR) circle [radius=3pt];
	\fill (vL) circle [radius=3pt];
\end{tikzpicture}

                \caption{The diagram $\Gamma_R$}
                \label{fig:splicingexpl:GammaR}
        \end{subfigure}     
        
        \end{tabular}

        \caption{Splicing a diagram $\Gamma$ along an edge $e$.} \label{fig:splicingexpl} \end{figure}
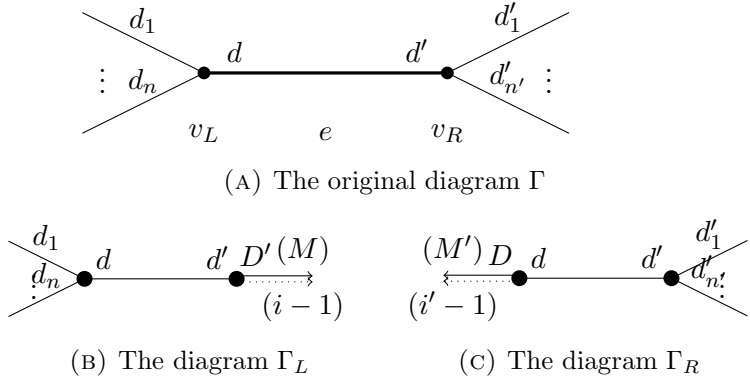 
\begin{definition}
 
Let $\Gamma$ be a realizable diagram. We   define the {\it motivic zeta function} of $\Gamma$ as
\begin{eqnarray*}
\motivicZeta{\Gamma} &:=& \sum_{v \text{ is a  node}} \frac{(\lefschetzmotive - 1) (\lefschetzmotive + 1 - \delta_v)
T^{N_v}}{\lefschetzmotive^{\nu_v} - T^{N_v}} \\
  & & + \sum_{e= (v,w) \text{ is an edge}} \frac{(\lefschetzmotive - 1)^2 T^{N_v + N_w}}{(\lefschetzmotive^{\nu_v} - T^{N_v})(\lefschetzmotive^{\nu_w} - T^{N_w})} \\
 & & + \sum_{a \text{ (dotted) arrow at $v$}} \frac{(\lefschetzmotive - 1)^2 T^{N_v + N_a}}{(\lefschetzmotive^{\nu_v} - T^{N_v})(\lefschetzmotive^{\nu_a} - T^{N_a})} 
 \in\CompletedGrothendieckRing{\C}[[T]].
\end{eqnarray*}
\end{definition}

It is possible to   define  motivic zeta functions for splice diagrams
using similar formulas  as described in \cite{MR1425327}. These formulas for the topological zeta function
are   used in \cite{NemethiVeys2012} to prove \eqref{eq:splicingformula:intro}. However 
it is unlikely that a proof of 
 Theorem \ref{theorem:splicedecompositionformulaformotiviczetafucntions} can be constructed using
 these formulas without  our notion of diagrams because of the complexity of the formulas.

A monodromic motivic zeta function is not available for our notion of diagrams. However we can define a twisted topological zeta function. Consider an $e\in \N$ and define
\begin{eqnarray*}
 \topZetaTwistedDiagram{\Gamma}{e}{s}  &:=&
 \sum_{\overset{ v \text{ is a  node},}{ e\mid N_v}} \frac{2-\delta_v}{ ( sN_v + \nu_v )} \\
  & & + \sum_{\overset{e= (v,w) \text{ is an edge},}{ e\mid N_v, e\mid N_w}} 
  \frac{1}{( sN_v + \nu_v )(s{N_w} + {\nu_w} )} \\
 & & + \sum_{\overset{a \text{ (dotted) arrow at $v$},}{ e\mid N_v, e\mid N_a}} 
 \frac{1}{(s {N_v} + {\nu_v} )(s{N_a} + {\nu_a})} 
 \in \Q(s)
\end{eqnarray*}
if $\Gamma$ is a realizable diagram.

\makeatletter{}

 \begin{figure}
  \centering
 \begin{subfigure}[b]{0.6\textwidth}
   \centering 
\begin{tikzpicture}
 \startpicturesplicingproof
 \partBpicturesplicingproof
 \partApicturesplicingproof
 \partCpicturesplicingproof
\end{tikzpicture}

    \caption{$\Gamma'$}
    \label{fig:Gamma}
        \end{subfigure}          
   \     
  
  \
  
        \begin{subfigure}[b]{0.6\textwidth}
                \centering  
  \begin{tikzpicture}
    \startpicturesplicingproof
 \partBpicturesplicingproof
 \partDpicturesplicingproof
 \partEpicturesplicingproof
  \end{tikzpicture}

                \caption{$\tilde{\Gamma}'$}
                \label{fig:GammaT}
        \end{subfigure}          
 
        \caption{
        The original diagram $\Gamma'$ and the new diagram $\tilde{\Gamma}'$,
        used in the proof of Theorem \ref{theorem:splicedecompositionformulaformotiviczetafucntions}         
        and consisting  of the same nodes and edges as the diagrams in Figure 
	  \ref{fig:diagramsoftheoremsplicedecompositionformulaformotiviczetafucntions1}. 
        \label{fig:diagramsoftheoremsplicedecompositionformulaformotiviczetafucntions2} 
        }
\end{figure}
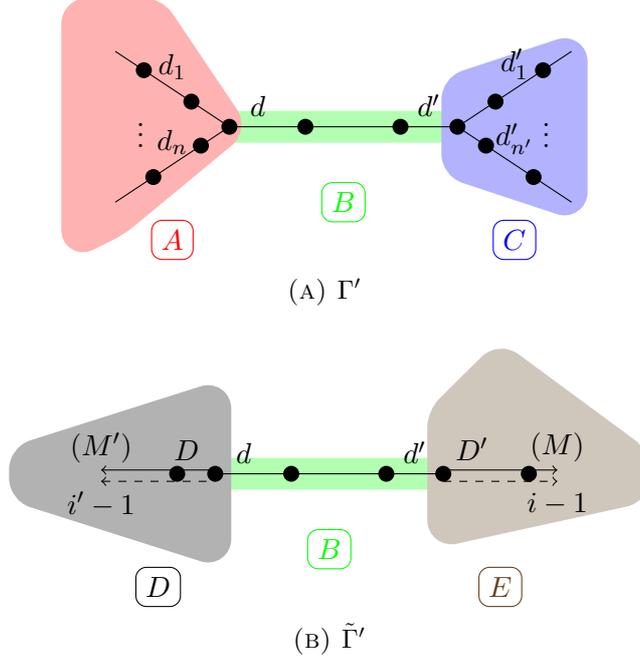
 
 \begin{figure}
  \centering

 \begin{subfigure}[b]{0.6\textwidth}
                \centering 
                           \begin{tikzpicture}
    \startpicturesplicingproof
 \partBpicturesplicingproof
 \partApicturesplicingproof
 \partEpicturesplicingproof
  \end{tikzpicture}
                \caption{$\Gamma_L'$}
                \label{fig:GammaL}
        \end{subfigure}  
        
   \     
  
  \
        
        \begin{subfigure}[b]{0.6\textwidth}
                \centering 
 \begin{tikzpicture}
 \startpicturesplicingproof
 \partBpicturesplicingproof
 \partDpicturesplicingproof
 \partCpicturesplicingproof
 \end{tikzpicture}
                \caption{$\Gamma_R'$}
                \label{fig:GammaR}
        \end{subfigure}
                \caption{
        The diagrams $\Gamma'_L$ and $\Gamma'_R$,
        used in the proof of Theorem \ref{theorem:splicedecompositionformulaformotiviczetafucntions}         
        and consisting of the same nodes and edges as the diagrams in Figure \ref{fig:diagramsoftheoremsplicedecompositionformulaformotiviczetafucntions2}.
        \label{fig:diagramsoftheoremsplicedecompositionformulaformotiviczetafucntions1} 
        }
        \end{figure}
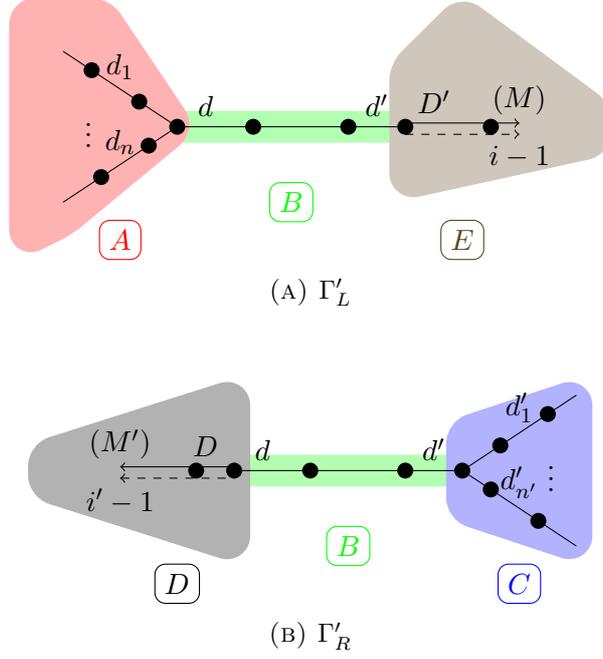

\section{Splicing formula for motivic zeta function}
 
In this section we will prove a splicing formula  for  motivic zeta functions for (splice) diagrams. This
immediately generalizes \cite[Theorem 3.2.4 (1)]{NemethiVeys2012}.

Consider a diagram $\Gamma$ with an edge $e$ between $v_L$ to $v_R$ as in Figure \ref{fig:splicingexpl:Gamma}.
Recall that $D = \prod_{i=1}^{n} d_i$ and $D' = \prod_{i=1}^{n'} d'_i$. We define the multiplicities
\begin{equation}\label{eq:splicing:multiplicityoff}
   M =  \sum_{a \text{ arrow at right side}} N_a l_{ea } \qquad \text{ and } \qquad    M' =  \sum_{a \text{ arrow at left side}} N_a l_{ea },     
\end{equation}
 and the multiplicities
\begin{equation}\label{eq:splicing:multiplicityofomega}
 i = \sum_{w \text{ node at right side}} (2-\delta_w)l_{ew} + \sum_{a \text{ dotted arrow at right side}} l_{ea}(i_a-1)
\end{equation}  and \begin{equation}\label{eq:splicing:multiplicityofomegap}
   i' = \sum_{w \text{ node at left side}} (2-\delta_w)l_{ew} + \sum_{a \text{ dotted arrow at left side}} l_{ea}(i_a-1),
\end{equation}
where $l_{ea}$ is the  product of the edge decorations adjacent to the path from $e$ to $a$ but not on it. This product
does not use the decorations on $e$ itself. 
We obtain
the diagram $\Gamma_L$ by removing all the nodes and edges at the side of $v_R$ and add
two arrows  and a node as in Figure \ref{fig:splicingexpl:GammaL}. Analogously we have $\Gamma_R$ as in Figure \ref{fig:splicingexpl:GammaR}.
This procedure is called splicing (of $\Gamma$ along $e$).

 We now state and prove our result.
\begin{theorem} \label{theorem:splicedecompositionformulaformotiviczetafucntions}
Consider a diagram  $\Gamma$  
and the splicing of   $\Gamma$ into $\Gamma_L$ and $\Gamma_R$. Then  we have
\[
\motivicZeta{\Gamma} = \motivicZeta{\Gamma_L} + \motivicZeta{\Gamma_R} - \frac{(\lefschetzmotive - 1)^2 T^{M + M'}}{ (\lefschetzmotive^i - T^M)(\lefschetzmotive^{i'} - T^{M'})}.
\]
\end{theorem}

 \begin{proof}
 Define   the diagram $\tilde{\Gamma}$ as follows:
 
\begin{center}
\makeatletter{}\begin{tikzpicture}
	\renewcommand*{\locationOfNode}{2}
	\renewcommand*{\locationOfStart}{4}
	\coordinate (vL) at (-\locationOfNode,0) ;
	\coordinate  (vR) at (\locationOfNode,0)  ;
	  \coordinate (vLL) at (-\locationOfStart,0) ;
	\coordinate (vRR) at (\locationOfStart,0)  ;
	\coordinate (labelup) at (0,.3);
	\coordinate (labeldown) at (0,-.3);
	
	\coordinate (down) at (0,-.1);
	\coordinate (up) at (0,.05);
	
	\coordinate (vLu) at ($(vL) + (up)$);
	\coordinate (vLLu) at ($(vLL) + (up)$);
	\coordinate (vRu) at ($(vR) + (up)$);
	\coordinate (vRRu) at ($(vRR) + (up)$);

	\coordinate (vLd) at ($(vL) + (down)$);
	\coordinate (vLLd) at ($(vLL) + (down)$);
	\coordinate (vRd) at ($(vR) + (down)$);
	\coordinate (vRRd) at ($(vRR) + (down)$);
	
	\draw[<-] (vLLu) -- (vLu) node[ near end, above]{$D$};
	\draw[<-, dashed] (vLLd) -- (vLd);
	\draw (vL) -- (vR)  node[very near start, above]{$d$} node[very near end, above]{$d'$} node[midway, below]{$\tilde{e}$};
	\draw[->] (vRu) -- (vRRu) node[ near start, above]{$D'$};
	\draw[->, dashed] (vRd) -- (vRRd);
	\fill (vR) circle [radius=3pt];
	\fill (vL) circle [radius=3pt];
	
	\node at ($(vLLu) + (labelup)$) {$(M')$};
	\node at ($(vRRu) + (labelup)$) {$(M)$};
	
	\node at ($(vLLd) + (labeldown)$) {$i'-1$};
	\node at ($(vRRd) + (labeldown)$) {$i-1$};
\end{tikzpicture} 
\end{center}
 where   $\tilde{e}$ is the edge between the nodes.
Remark that this is the dual graph of a (toric) resolution of the polynomial $f = x^My^{M'}$ and the differential form
$\omega = x^{i-1}y^{i'-1} {dx} \wedge {dy}$. Hence $\motivicZeta{\tilde{\Gamma}} =  \frac{(\lefschetzmotive - 1)^2 T^{M + M'}}{ (\lefschetzmotive^i - T^M)(\lefschetzmotive^{i'} - T^{M'})}$ 
and thus we need to show that
\begin{equation}\label{eq:relbetweenmotivc}
\motivicZeta{\Gamma} + \motivicZeta{\tilde{\Gamma}} = \motivicZeta{\Gamma_L} + \motivicZeta{\Gamma_R}.
\end{equation}
We prove this by considering suitable realizable refinements of these diagrams. 
First, take a realizable refinement $\Gamma'$ of $\Gamma$. This is drawn in Figure \ref{fig:Gamma},
where you have the division into parts \letterA, {\letterB} and \letterC. Remark that the edges crossing the border belong to \letterB. 

Second, take a realizable refinement $\tilde{\Gamma}'$ of $\tilde{\Gamma}$ such that the edge $e$ of $\Gamma$ has the same refinement 
 as the edge $\tilde{e}$ of $\tilde{\Gamma}$. Hence we can consider
Figure \ref{fig:GammaT} where we have a division into \letterD, {\letterB} and {\letterE} where the subdiagram {\letterB} is the same as in Figure \ref{fig:Gamma}.

We glue these refinements  together to obtain refinements $\Gamma_L'$ and $\Gamma_R'$ of
$\Gamma_L$ and $\Gamma_R$. This is shown in Figures \ref{fig:GammaL} and \ref{fig:GammaR}.

We compare the diagrams $\Gamma'$ and $\tilde{\Gamma}'$ to the  
diagrams $\Gamma_R'$ and $\Gamma_L'$. Both groups contain 
  the subdiagrams \letterA, \letterB, \letterB, \letterC, {\letterD} and  \letterE.
  Hence if we calculate the motivic zeta function, where we take the sum over all nodes and edges in these subdiagrams, 
  we have
  proven (\ref{eq:relbetweenmotivc}) 
  if we show that the corresponding nodes have the same multiplicities.
But this is easily seen by using (\ref{eq:formulafornodes}), (\ref{eq:formulafordifferentialform}), (\ref{eq:splicing:multiplicityoff}), (\ref{eq:splicing:multiplicityofomega}) and
(\ref{eq:splicing:multiplicityofomegap}).	
 \end{proof}

 We can connect our proof to the splicing of links. Recall that splicing consists of taking two links, where in each
 a knot is selected, removing a tubular neighborhood around these knots and glueing the remainders together. However if
 you glue  these removed tubular neighbourhoods together, you find the link of $\Gamma'$.

\makeatletter{}\section{Algebraic dependence of the monodromic motivic zeta function}
Following Theorem \ref{theorem:splicedecompositionformulaformotiviczetafucntions}, we want to define
a monodromic zeta function for a diagram. But it turns out that this is not possible. 
Consider for this the family of  polynomials
\[
  f_\lambda = xy^2(x-y)(x-\lambda y) \in \C[x,y],
\] where $\lambda \in \C\setminus \{0,1\}$. The splice diagram associated to this family of 
polynomials is independent of $\lambda$. But we do have the following result.
\begin{proposition} \label{proposition:algebraicdependenceofmonodromiczetafunctions}
 There exist  $\lambda, \lambda' \in \C\setminus\{0,1\}$ such that
 \[
  \monodromicmotivicZetaFunction{f_\lambda} \neq \monodromicmotivicZetaFunction{f_{\lambda'}}.
 \]

\end{proposition}
This results shows that a monodromic motivic zeta function cannot be defined for a diagram
since the diagrams for this  family of polynomials are the same.

\begin{itemize}
 \item 
 Define the Grothendieck group of abelian varieties $\Grothendieckgroupofabelianvarieties{\C}$ as the 
 abelian group generated by isomorphism classes of abelian varieties with the relations 
 $\classof{A \sumofabelian B } = \classof{ A} + \classof{B} $ and
\item define $\Grothendieckgroupofabelianvarietiesuptoisogeny{\C}$ as the abelian group generated by
isogeny classes of abelian varieties and relations $\classof{A\sumofabelian B} = \classof{A} + \classof{B} $.
\end{itemize}
The structure of $\Grothendieckgroupofabelianvarietiesuptoisogeny{\C}$ is easier since Poincar\'e's complete 
irreducibility theorem 
\cite[p. 173]{mumford1974abelian}   implies that $\Grothendieckgroupofabelianvarietiesuptoisogeny{\C}$
 is isomorphic to the free abelian group on simple abelian varieties. Hence equality of the classes 
 of two abelian varieties in 
 $\Grothendieckgroupofabelianvarietiesuptoisogeny{\C}$ implies that they are isogenous. 
 
 In the appendix we describe a group morphism $\widetilde{\Pic} :  \CompletedGrothendieckRing{\C} \to \Grothendieckgroupofabelianvarieties{\C} $.
 This morphism sends the class of a smooth complete variety to the class of its Jacobian. We will use this morphism in the 
 following proof, where we compose it with the forgetful morphism 
 $\Grothendieckgroupofabelianvarieties{\C} \to \Grothendieckgroupofabelianvarietiesuptoisogeny{\C}$.

\begin{proof}[Proof of Proposition \ref{proposition:algebraicdependenceofmonodromiczetafunctions}] Let 
$\pi : X \to \affinenspace{2}{\C}$ be
the blowup at the origin.
 This is an embedded resolution of singularities for $f_\lambda$. The formula of Denef-Loeser then tells us that
 we need to prove that the class of 
 \[
  \tilde{ E}^\circ_\lambda = \{(x,y) \in \C^2 \mid y^5 = x(x-1)(x-\lambda), y \neq 0\}
 \] does depend on $\lambda$. 
 
 We will now prove that if two varieties are equal in $\CompletedGrothendieckRing{\C}$, then their jacobians are isogenous. 
 Combining
 this with \cite[Proposition 2.7]{dejongNoot} we find that the set 
 $\{\classof{\tilde{ E}^\circ_\lambda} \in \CompletedGrothendieckRing{\C} \mid \lambda \in \C \setminus \{0,1\} \}$ is infinite. 

We have a map 
 $\CompletedGrothendieckRing{\C} \stackrel{\tilde\Pic}{\to}  \Grothendieckgroupofabelianvarieties{\C} \to \Grothendieckgroupofabelianvarietiesuptoisogeny{\C}$ 
 which sends a complete smooth variety to the class of its jacobian. Using the remark before the start of the proof, we find that 
 equality in $\LocalizedGrothendieckring{\C}$ implies that their jacobians are isogenous.
\end{proof}

This has implications on how to formulate a splicing formula for the monodromic motivic zeta functions. The same proof of 
Theorem \ref{theorem:splicedecompositionformulaformotiviczetafucntions}
will work if you work with a more general notion of a diagram. In this generalization you need to encode the information
of $\EIOT$ more carefully, for example by remembering the locations of the branch  points of the cover 
$\EIOT \to \projectivenspace{1}{\C}$.

Another way is to formulate the splicing formula in terms of $f$ and $\omega$.
\begin{theorem}
Let $f\in \C[x,y]$ and let $\omega$  be a differential form. Also fix an edge $e$ in  an associated diagram $\Gamma$ for some embedded resolution and
denote by $\Gamma_L$ and $\Gamma_R$ the resulting diagrams after splicing along $e$.
Then there exists $f_L, f_R \in \C[x,y]$ and differential forms $\omega_L, \omega_R$, whose splice diagrams are $\Gamma_L$ and $\Gamma_R$
for some embedded resolutions, such that
\[
 \monodromicmotivicZetaFunction{f, \omega} = \monodromicmotivicZetaFunction{f_1, \omega_1} + \monodromicmotivicZetaFunction{f_2, \omega_2} -  
 \frac{(\lefschetzmotive - 1)^2 T^{M + M'}}{ (\lefschetzmotive^i - T^M)(\lefschetzmotive^{i'} - T^{M'})}.
\]  

\end{theorem}

\makeatletter{}\section{Monodromy conjecture\label{section:mc}}
In this section we discuss how we can lift  the results in \cite{NemethiVeys2012} about the generalized monodromy conjecture 
to the motivic level.
Recall that we consider our  polynomial $f$ as a germ $(\affinenspace{2}{\C},0) \to (\affinenspace{1}{\C},0)$. Let  $F_0$
be the Milnor fiber of this germ, $h_i : H_i(F_0, \C) \to H_i(F_0, \C)$ the algebraic monodromy ($i=0,1$), 
$\Delta_i = \det(tI - h_i)$
the characteristic polynomial of $h_i$ and $\zeta(t) = \Delta_1\slash \Delta_0$.

Given a diagram $\Gamma$, we define the formula
\[
 \zeta_\Gamma(t) = \prod_{ {v \text{ is a  node}}}  \left( t^{N_v} - 1   \right)^{\delta'_v}
\]
where $\delta'_v$ is the valency of $v$ considered in the graph including the arrows corresponding to strict transforms of $f$, but  without the dotted arrows corresponding to the strict transforms of $\supp \omega$. 
This recovers
the monodromy zeta function of $f$. Remark that $\Delta_0 = t^d-1$
where $d = \gcd_{a \text{ is an arrow}} N_a $, and thus we can discuss monodromy eigenvalues of a diagram. 

In what follows it will be implicit
that monodromy eigenvalues are the monodromy eigenvalues of our fixed $f$.

\subsection{Motivic zeta function}

N\'emethi and Veys defined a differential form $\omega$ on a diagram $\Gamma $ to be allowed if the 
following conditions are satisfied:
\begin{itemize}
 \item $(N_a,i_a) \neq (0,0)$  for all (dotted) arrows.
 \item each star-shaped subdiagram with the induced decorations obtained after repeated splicing needs to be 
 allowed.
 \item if $\Gamma$ is a star-shaped diagram the following condition needs to be satisfied:
 
 \begin{flushright}
\begin{minipage}{0.9\textwidth}
 Let the central node be connected to $n$ boundary vertices whose supporting edges have decorations $\{d_l\}_{l=1}^n$,
and with $r$ other incident edges connecting with arrowheads. Then the decorations $i_1-1, \ldots, i_n-1$
of the dashed arrows at these boundary vertices are subject to the following restrictions:
\begin{center}\begin{minipage}{0.8\textwidth}
 If $d_l\vert i_l$ for at least $n+r-2$ indices $l\in\{1,\ldots, n\}$, then $i_l = d_l$
 for at least $n+r-2$ of the indices $l$.
\end{minipage}
\end{center}
\end{minipage}
 \end{flushright}
 \vskip 5mm
 
 \begin{center}
  \makeatletter{}

\begin{tikzpicture}
	\newcommand*{\distu}{0.8}
	\newcommand*{\distr}{3}
	\coordinate (RU) at (\distr,\distu) ;
	\coordinate (RD) at (\distr,-\distu)  ;
	\coordinate (LU) at (-\distr,\distu);
	\coordinate (LD) at (-\distr,-\distu);
	\coordinate (MID) at (0,0);
	\draw[thick] (MID)--(RU) node[near start, above ] {$d_1$};
	\draw[thick] (MID)--(RD) node[near start, below ] {$d_n$};
	\draw[->, thick] (MID)--(LU);
	\draw[->, thick] (MID)--(LD);
	\draw[->, thick, dotted] (RD)--($(RD) + (1,0)$);	\draw node at ($(RD) + (1.7,0)$) {$i_1-1$} ;
	\draw node at ($(RU) + (1.7,0)$) {$i_n-1$} ;
	\draw[->, thick, dotted] (RU)--($(RU) + (1,0)$);	\draw node at (1.5,0.1){$\vdots$} ;  
	\draw node at (-1.5,0.1){$\vdots$};  
	\draw node at (-\distr, -2) {$r$ arrowheads };
	\draw node at (-\distr, -2.5) {which might be doublearrows};
	\draw node at (\distr, -2) {$n$ boundary vertices};
 	\fill (MID) circle [radius=3pt];
 	\fill (RD) circle [radius=3pt];
 	\fill (RU) circle [radius=3pt];
\end{tikzpicture}

 \end{center}

\end{itemize}
N\'emethi and Veys showed that this set of allowed forms has the following properties:
 \begin{itemize}
  \item for every allowed form $\omega$ and every pole $s_0$ of $\topZeta{f,\omega}{s}$, $\exp(2\pi i s_0)$
  is a monodromy eigenvalue. 
  \item $dx\wedge dy$ is allowed;
  \item for every   monodromy eigenvalue $\lambda$, there is
  an allowed form $\omega $ such that  $\topZeta{f,\omega}{s}$ has
  a pole $s_0$ and $\lambda=\exp(2\pi i s_0)$. 
 \end{itemize}
To state the generalized monodromy conjecture for   the motivic zeta function we need the notion of a pole. This has 
been done by Rodrigues
and Veys in \cite{MR1978573}. We will use a more direct approach.
\begin{definition}
 We call $s \in \Q$ a pole of  $\motivicZetaWithDifferentialForm{f}{\omega}$ if there exists no set 
 $U \subseteq \Z_{\geq 0}\times \Z_{\geq 0} $ such that
 \[
  \motivicZetaWithDifferentialForm{f}{\omega} \in 
 \CompletedGrothendieckRing{\C}\left[\frac{T^N}{\lefschetzmotive^\nu - T^N}\right]_{(v,N) \in U} \subseteq  \CompletedGrothendieckRing{\C}[[T]].
 \]
and such that $s\neq -\frac{\nu}{N}$ for all $(\nu,N) \in U$.
\end{definition}
 
Because $\CompletedGrothendieckRing{\C}$ is not a domain \cite{Ekedahl2009}, we are careful in formulating the next theorem. 
\begin{theorem}\label{thm:generalizedmotivicmonodromyconjecture}
 Let $f \in \C[x,y]$ and $\omega$ an allowed form.  
Then 
\[
 \motivicZetaWithDifferentialForm{f}{\omega} \in 
 \CompletedGrothendieckRing{\C}\left[\frac{T^N}{\lefschetzmotive^\nu - T^N}\right]_{(v,N) \in S} \subseteq  \CompletedGrothendieckRing{\C}[[T]].
\]
where $S = \{(\nu,N) \in \Z_{\geq 0}\times \Z_{\geq 0} \mid (\nu,N) \neq (0,0), \exp( - 2\pi i\frac{\nu}{N}) \text{ is a monodromy eigenvalue of }f\}$.
\end{theorem}
\begin{proof}
 The proof goes as in the case of the topological zeta function, where you need  first to consider 
 star-shaped realizable diagrams. In this case you need to prove that $\alpha_i= -N_i\frac{\nu}{N} + \nu_i = 1$  for 
 sufficiently many $i$.
 
 The general case can be proved in the same way as  in the proof for the topological zeta function.
\end{proof}

By specializing to the topological zeta function, we easily find that every  monodromy eigenvalue is obtained as a pole.
Hence we have proven the generalized monodromy conjecture  for the motivic zeta function.
\begin{corollary}\label{cor:generalizedmonodromyconjectureformotiviczetafunctions}
Consider the set of allowed forms for a diagram $\Gamma$ of $f \in \C[x,y]$. It
satisfies the following conditions:
 \begin{itemize}
  \item for every allowed form $\omega$, every pole of $ \motivicZetaWithDifferentialForm{f}{\omega}$ induces 
  a monodromy eigenvalue of $f$. More specifically
  Theorem \ref{thm:generalizedmotivicmonodromyconjecture} holds.
  \item $dx\wedge dy$ is allowed;
  \item every monodromy eigenvalue is obtained as a pole of the motivic zeta function of $f$ with respect to $\omega$.
 \end{itemize}

\end{corollary}

\subsection{Monodromic motivic zeta function}
Theorem \ref{thm:generalizedmotivicmonodromyconjecture} however does not generalize to the case of the monodromic motivic zeta
function.
This is due to the fact that an allowed form is made such that the residues of candidate `bad' poles of the topological
zeta functions
are zero. We give here examples of twisted topological zeta functions because a pole of a twisted topological zeta  
function will also be a pole of the monodromic motivic zeta function.

We give an example where the twisted topological zeta function has a pole which does not induce a monodromy eigenvalue, and 
thus an analogue of  Corollary \ref{cor:generalizedmonodromyconjectureformotiviczetafunctions} cannot hold.

\begin{example} 
Consider the cusp $f=x^3 - y^2$ , the minimal resolution $\pi$ and differential form 
 $\omega= x^3y^3dx\wedge dy$. Then its splice diagram is the following:
  \begin{center}
 \makeatletter{}
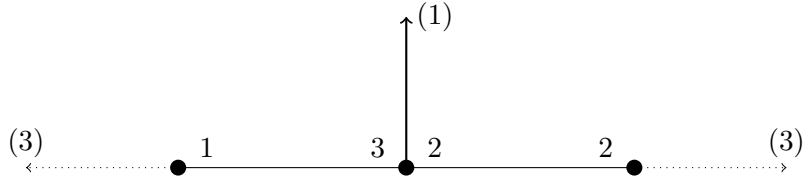
\begin{tikzpicture}
	\newcommand*{\locationOfLeft}{3}
	\newcommand*{\heightOfArrow}{2}
	\renewcommand*{\locationOfStart}{4}
	\renewcommand*{\heigthofedges}{1}
	\coordinate (een) at (-\locationOfLeft, 0);
	\coordinate (twee) at (0, 0);
	\coordinate (drie) at (\locationOfLeft, 0);
	\coordinate (vier) at (0, \heightOfArrow);
	\coordinate (vijf) at ( -\locationOfLeft-\heightOfArrow,0);
	\coordinate (zes) at ( +\locationOfLeft+\heightOfArrow,0);
	\draw (een)--(twee) node[very near start, above]{$1$} node[very near end, above]{$3$};
	\draw (twee)--(drie) node[very near start, above]{$2$} node[very near end, above]{$2$};
	\fill (een) circle [radius=3pt];
	\fill (twee) circle [radius=3pt];
	\fill (drie) circle [radius=3pt];
	\draw[thick,->] (twee)--(vier);
	\draw node[right] at (vier) {(1)};
	\draw node[above] at (vijf) {(3)};
	\draw node[above] at (zes) {(3)};
	\draw[dotted, ->] (een)--(vijf);
        \draw[dotted, ->] (drie)--(zes);
	
\end{tikzpicture} 
 \end{center}
 This form $\omega$ is allowed but   we find that 
 \[
  \topZetaTwisted{f}{\omega}{6}{s} = -\frac{1}{6s + 21},
 \] which has $s_0= -\frac{7}{2}$ as a pole. But $\exp(2\pi i s_0) = -1$ is not a  monodromy eigenvalue.  
\end{example}
One can wonder if a suitable subset of allowed forms can achieve
Corollary \ref{cor:generalizedmonodromyconjectureformotiviczetafunctions}. 
The following example shows however that we cannot obtain all poles.
\begin{example}
 Consider the polynomial $f= (y^3 - x^4)^5 + x^2y^{15}$. We have the following diagram $\Gamma$
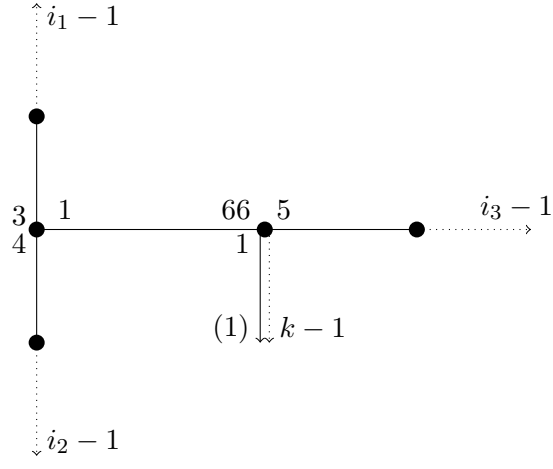
   \begin{center}
 \makeatletter{}\begin{tikzpicture}
\coordinate (een) at (0,1.5);
\coordinate (twee) at (0,0);
\coordinate (drie) at (0,-1.5);
\coordinate (vier) at (3,0);
\coordinate (vijf) at (5,0);
\coordinate (zes) at (3 ,-1.5);
\coordinate (zeven) at (0,3);
\coordinate (acht) at (0,-3);
\coordinate (negen) at (6.5,0);
\coordinate (diff) at (.06,0);
\fill (een) circle [radius=3pt];
\fill (twee) circle [radius=3pt];
\fill (drie) circle [radius=3pt];
\fill (vier) circle [radius=3pt];
\fill (vijf) circle [radius=3pt]; 
\draw (een)--(twee) node[very near end, above, left]{$3$};
\draw (drie)--(twee) node[very near end,  left]{$4$};
\draw (twee)--(vier) node[very near end, above]{$66$} node[very near start, above]{$1$};
\draw (vier)--(vijf) node[very near start, above]{$5$};
\draw[->] ($(vier) - (diff)$)--($(zes) - (diff)$) node[very near start, left]{$1$} node[very near end, left]{$(1)$};
\draw[->, dotted] (een)--(zeven) node[very near end, right]{$i_1-1$};
\draw[->, dotted] (drie)--(acht) node[very near end, right]{$i_2-1$};
\draw[->, dotted] ($(vier) + (diff)$)--($(zes) + (diff)$) node[very near end, right]{$k-1$};
\draw[->, dotted] (vijf)--(negen) node[very near end, above]{$i_3-1$};

\end{tikzpicture} ,
 \end{center}
 where we are already considering some form with $i_1, i_2, i_3, k \in \Z$. The monodromy zeta function is
 \[
  \zeta(T) = \frac{(T^{330} - 1)(T^{60} - 1)}{(T^{66} - 1)(T^{15} - 1)(T^{20} - 1)},
 \]which implies that $\lambda = \exp(2\pi i \frac{1}{110})$ is a monodromy eigenvalue. 
However there exist no  $i_1, i_2, i_3, k$ such that $\topZetaTwistedDiagram{\Gamma}{330}{s}$ has a pole $s_0$ with
the condition $\lambda = \exp(2\pi i s_0)$ and such that $\lambda' = \exp(2\pi i s'_0) $ is a monodromy eigenvalue whenever
$s_0'$ is the pole of $ \topZetaTwistedDiagram{\Gamma}{60}{s}$. Indeed, if $\lambda$ is a pole, then $s_0$ needs
to be a pole of $\topZetaTwistedDiagram{\Gamma}{330}{s}$. This implies that 
$ 2 i_1 + 3i_2 \equiv 3 \pmod{ 6 }$. But now the pole of $ \topZetaTwistedDiagram{\Gamma}{60}{s}$ will not
induce a monodromy eigenvalue since $i_1 \equiv 0 \pmod 3$ and $i_2 \equiv 1 \pmod 2$.
\end{example}

{} 

 \bibliographystyle{amsplain}

 \appendix
 
\makeatletter{}

\section{Existence of the Picard morphism}

It turns out that the class of a smooth and proper algebraic variety in the Grothendieck ring of varieties
 determines the class of its Picard scheme.  
 In this appendix we will define and prove this statement
using the argument described in \cite{Ekedahl2009}. We do this since \cite{Ekedahl2009} was never published and to clarify several
steps in his argument.
 
 Let $X$ be a smooth and proper algebraic variety over $\C$. 
Recall that the Picard functor $\Pic_{X\slash \C}$ is a functor from the category of locally Noetherian $\C$-schemes
to the category of abelian groups defined by the formula 
\[
 \Pic_{X\slash \C}(T) = \Pic(X_T) \slash \Pic(T)
\] where $X_T = X \times_\C T$. It turns out that the associated fppf sheaf is representable by an abelian 
scheme whose identity component is an abelian variety and whose group of geometric components is finitely generated.
This scheme will be denoted by $\Pic(X)$. See \cite[Part 5]{fantechi2005fundamental} for more information on 
the Picard scheme.

\begin{definition}Define  $ \Grothendieckgroupofabelianschemes{\C}$ as 
  the abelian group generated by isomorphism classes of commutative algebraic group schemes over $\C$ whose identity component is 
 an abelian variety and whose group of geometric components is finitely generated, subject to the 
 relations $\classof{A \sumofabelian B } = \classof{ A} + \classof{B} $.

 \end{definition}
This leads us to the main result.
\begin{theorem}\label{theorem:existenceofPicardscheme}
 There is a (unique) group homomorphism $\Pic  :  \Grothendieckring{\C} \to \Grothendieckgroupofabelianschemes{\C}$ such
 that  $\Pic ( \classof{X}) = \classof{\Pic(X)}$ for every smooth proper variety $X$, 
 and it extends to a morphism 
 $\Pic : \CompletedGrothendieckRing{\C} \to \Grothendieckgroupofabelianschemes{\C}$.
\end{theorem}

This theorem provides us with a new technique to compare elements in the Grothendieck ring. 

 Recall that the Grothendieck group of abelian varieties $\Grothendieckgroupofabelianvarieties{\C}$ is defined as the 
 abelian group generated by isomorphism classes of abelian varieties with the relations 
 $\classof{A \sumofabelian B } = \classof{ A} + \classof{B} $. Using this theorem,
 we find the existence of $\widetilde{\Pic} :  \CompletedGrothendieckRing{\C} \to \Grothendieckgroupofabelianvarieties{\C} $ which
 sends a smooth complete variety to the class of its jacobian. It is obtained 
 by composing   $\Pic$ and the morphism $\Grothendieckgroupofabelianschemes{\C} \to  \Grothendieckgroupofabelianvarieties{\C} $, where this last
  map is defined by sending the class of a commutative algebraic group scheme (whose identity component is an abelian variety) to the class
  of his identity component.

  The keystone of the proof is Bittner's presentation of $ \Grothendieckring{\C} $, which we restate here.

\begin{theorem*}{\cite[Theorem 3.1]{MR2059227} }
 The Grothendieck group of $\C$-varieties $\Grothendieckring{\C}$ is isomorphic  
  to the abelian group generated by the isomorphism classes of smooth projective $\C$-varieties subject to the relations
  $\classof{\emptyset} = 0 $ and $\classof{\Bl_YX} - \classof{E} = \classof{X} - \classof{Y}$, where $X$ is
  smooth and projective, $Y\subset X$ is a closed smooth subvariety, $\Bl_Y X$ is the blow-up of $X$ along $Y$ and $E$
  is the exceptional divisor of this blow-up.
 
\end{theorem*}
This implies the following presentation of $\LocalizedGrothendieckring{\C} = \Grothendieckring{\C}[\lefschetzmotive^{-1}]$:
\begin{center}
\begin{minipage}{0.8\textwidth}
  $\LocalizedGrothendieckring{\C}$ is
generated by the elements $\quotientInGrothendieckRing{\classof{X}}{\lefschetzmotive^n}$
 where $X$ is smooth and projective, subject to the relations

\[\quotientInGrothendieckRing{\classof{X\times \projectivenspace{1}{\C}} }{\lefschetzmotive^{n+1} } 
 = \quotientInGrothendieckRing{\classof{X} }{\lefschetzmotive^{n} } + 	
 \quotientInGrothendieckRing{\classof{X }}{\lefschetzmotive^{n+1} }, \]  where $X$ is smooth and projective, and 
 the relations $ \quotientInGrothendieckRing{\classof{\Bl_YX}}{\lefschetzmotive^n} - \quotientInGrothendieckRing{\classof{E}}{\lefschetzmotive^n} = \quotientInGrothendieckRing{\classof{X}}{\lefschetzmotive^n} - \quotientInGrothendieckRing{\classof{Y}}{\lefschetzmotive^n}$, where $X$, $Y$ and $E$ are as in the theorem.
 
\end{minipage}
\end{center}

\begin{proof}[Proof of Theorem \ref{theorem:existenceofPicardscheme}]
\begin{itemize}
 \item We first show that the morphism defined by $\Pic( \classof{X}) = \classof{\Pic(X)}$ for $X$ smooth and proper is well-defined as a map
 from $\Grothendieckring{\C}$ to $ \Grothendieckgroupofabelianschemes{\C}$. Using the Bittner representation,
 we need to show that 
  \[\Pic(\Bl_YX) \sumofabelian \Pic(Y) = \Pic(X) \sumofabelian \Pic(E)\] where $X$, $Y$ and $E$ are as in the theorem.
  Hence we need to show this on the level 
  of associated fppf sheaves. But   a blow-up is preserved under base change by a flat morphism and thus
   $\Pic(\Bl_{Y_T} X_T) = \Pic(X_T ) \sumofabelian \Z$ and $\Pic(E_T) = \Pic( Y_T ) \sumofabelian \Z$ for any flat morphism $T \to \C$
   which induces the wanted isomorphism \cite[Exercises 7.9 and 8.5 on pages 170 and 188]{MR0463157}.

 \item The next step is to define 
 $\Pic' : \LocalizedGrothendieckring{\C} \to  \Grothendieckgroupofabelianschemes{\C}$ and prove
 that this is actually an extension of $\Pic$.

 Define $\Pic'(\classof{X} \slash \lefschetzmotive^n )$
  as 
 \[
\WeilintermediateJacobian{n}{X} \sumofabelian \componentgroup{n}{X}
 \]   for a projective and smooth variety $X$ and $n \in \N$,  
 where
 \begin{itemize}
   \item $\componentgroup{n}{X}$ is the inverse image of the classes of type $(n+1, n+1)$ 
  under the map
  $\cohomologygroups{2n+2}{X}{\Z} \to \cohomologygroups{2n+2}{X}{\C}$ and
  \item $\WeilintermediateJacobian{n}{X}$ is the Weil intermediate Jacobian associated to the Hodge structure
    on $\cohomologygroups{2n+1}{X}{\Z}$. See \cite{1952} for more information.
 \end{itemize}
 As discussed, we need to verify the two types of relations. Consider $X$ to be a projective smooth $\C$-variety.
 \begin{itemize}
  \item First we  verify the blow-up relations. This is a consequence of  \cite[Theorem 7.31]{voisin2002hodge} and 
its proof since it induces that 
\[
 \cohomologygroups{k}{X}{\Z} \sumofhodgestructures 
 \left( \bigsumofhodgestructures_{i=0}^{r-2} \cohomologygroups{k-2i - 2}{Y}{\Z} \right) \sumofhodgestructures 
 \cohomologygroups{k}{Y}{\Z}
\]
is both isomorphic to $\cohomologygroups{k}{\Bl_Y X}{\Z} \sumofhodgestructures \cohomologygroups{k}{Y}{\Z}$ and
$\cohomologygroups{k}{X 	}{\Z} \sumofabelian \cohomologygroups{k}{E}{\Z}$ as Hodge structures, where $r$ is the
codimension of $Y$ in $X$.
\item  Remark  that the cup-products  map  \cite[Theorem 11.38]{voisin2002hodge}
for $X$  and $ \projectivenspace{1}{\C}$ induces  
\[
\cohomologygroups{i+2}{X \times \projectivenspace{1}{\C} }{\Z} \cong \cohomologygroups{i+2}{X}{\Z} \sumofhodgestructures\cohomologygroups{i}{X}{\Z} 
\]  as Hodge structures \cite[p. 32]{MR2050072}.
 \end{itemize}
 Since the relations hold on the level of Hodge structures, they also hold for $\Pic'$ and thus $\Pic'$ is 
 well-defined.

\item We show now   that $\Pic'$ is indeed an extension of $\Pic$
and thus we will
show that $\Pic(X) \cong \WeilintermediateJacobian{0}{X} \sumofabelian \componentgroup{0}{X}$.
Remark   that the connected component $\Pic^0(X)$ is the classical Jacobian of $X$ and thus isomorphic to
$\WeilintermediateJacobian{0}{X}$. 

The N\'eron-Severi group $\NeronSeveriGroup{X}$ is   defined by the short exact sequence
\[
 0 \to \Pic^0(X) \to \Pic (X) \to \NS(X) \to 0
\] and is the group of components of $\Pic(X)$. This group can also be identified with the image of 
$ d : \cohomologygroups{1}{X}{\mathcal{O}_X^\times} \to \cohomologygroups{2}{X}{\Z}$. 
Now Lefschetz's theorem on $(1,1)$-classes \cite[p.163]{griffiths2011principles} implies that $\NS(X)$ is exactly 
$\componentgroup{0}{X}$.

Since we are working over an algebraically closed field and $\NS(X)$ is a discrete finitely generated  abelian group,
the short exact sequence splits and thus $\Pic(X) \cong \Pic^0(X) \sumofabelian \NS(X)$.
\item To conclude we remark that $\cohomologygroups{n}{X}{\Z} = 0$ if $n > 2 \dimension{X}$ and thus 
$\Pic'{\quotientInGrothendieckRing{\classof{X}}{\lefschetzmotive^n}}  = 0$ if $\dim(X) -n \leq 0 $.
This implies that $\Pic'$ sends
every element of $F^0$ to $0$ and thus $\Pic' $ can be extended to $\CompletedGrothendieckRing{\C}$.
\qedhere 
\end{itemize}
\end{proof}
One of the results Ekedahl obtains with this is the fact that $\CompletedGrothendieckRing{\C}$  is not a domain.
 
\bibliography{ref} 

\end{document}